\documentclass[11pt, notitlepage]{article}
\usepackage{amssymb,amsmath,comment}
\catcode`\@=11 \@addtoreset{equation}{section}
\def\thesection{\arabic{section}}

\def\theequation{\thesection.\arabic{equation}}
\catcode`\@=12
\usepackage{colortbl}
\usepackage[mathscr]{eucal}
\usepackage{epsf}
\usepackage{esint}
\usepackage{a4wide}
\def\R{\mathbb{R}}

\DeclareMathOperator*{\esssup}{ess\,sup}
\DeclareMathOperator*{\essinf}{ess\,inf}

\DeclareMathOperator*{\osc}{osc}
\DeclareMathOperator*{\supp}{supp}

\newcommand{\noi} {\noindent}

\newcommand{\Tail} {\mathrm{Tail}}
\newcommand{\loc} {\mathrm{loc}}

\usepackage[all]{xy}
\catcode`\@=11
\def\theequation{\@arabic{\c@section}.\@arabic{\c@equation}}
\catcode`\@=12

\newtheorem{Theorem}{Theorem}[section]
\newtheorem{Lemma}[Theorem]{Lemma}

\newtheorem{Corollary}[Theorem]{Corollary}
\newtheorem{Remark}[Theorem]{Remark}
\newtheorem{Definition}[Theorem]{Definition}

\begin{document}

{\vspace{0.01in}}

\title
{\sc On the regularity theory for mixed local and nonlocal quasilinear elliptic equations}
\author{Prashanta Garain and Juha Kinnunen}

\maketitle

\begin{abstract}
We consider a combination of local and nonlocal $p$-Laplace equations and discuss several regularity properties of weak solutions.
More precisely, we establish local boundedness of weak subsolutions, local H\"older continuity of weak solutions, Harnack inequality for weak solutions and weak Harnack inequality for weak supersolutions.
We also discuss lower semicontinuity of weak supersolutions as well as upper semicontinuity of weak subsolutions.
Our approach is purely analytic and it is based on the De Giorgi-Nash-Moser theory, the expansion of positivity and estimates involving a tail term. 
The main results apply to sign changing solutions and capture both local and nonlocal features of the equation.
\medskip

\noi {Keywords: Regularity, mixed local and nonlocal $p$-Laplace equation, local boundedness, H\"older continuity, Harnack inequality, weak Harnack inequality, lower semicontinuity, energy estimates, De Giorgi-Nash-Moser theory, expansion of positivity}.

\medskip

\noi{\textit{2010 Mathematics Subject Classification:} 35B45, 35B65, 35D30, 35J92, 35R11}.
\end{abstract}

\section{Introduction}
In this article, we develop regularity theory of weak solutions for the problem
\begin{equation}\label{maineqn}
-\Delta_p u+\mathcal{L}(u)=0\text{ in }\Omega,
\quad 1<p<\infty,
\end{equation}
where $\Omega$ is a bounded domain in $\mathbb{R}^n$. The local $p$-Laplace operator is defined by
\begin{equation}\label{plap}
\Delta_p u=\text{div}(|\nabla u|^{p-2}\nabla u),
\end{equation}
and $\mathcal{L}$ is the nonlocal $p$-Laplace operator given by
\begin{equation}\label{fracplap}
\mathcal{L}(u)(x)=\text{P.V.}\int_{\mathbb{R}^n}|u(x)-u(y)|^{p-2}(u(x)-u(y))K(x,y)\,dy,
\end{equation}
where P.V. denotes the principal value. Here $K$ is a symmetric kernel in $x$ and $y$ such that
\begin{equation}\label{kernel}
\frac{\Lambda^{-1}}{|x-y|^{n+ps}}\leq K(x,y)\leq \frac{\Lambda}{|x-y|^{n+ps}}, 
\end{equation}
for some $\Lambda\geq 1$ and $0<s<1$.
Note that the $p$-Laplace operator $\Delta_p$ reduces to the classical Laplace operator $\Delta$ for $p=2$.
When $K(x,y)=|x-y|^{-(n+ps)}$, the operator $\mathcal{L}$ is the fractional $p$-Laplace operator $(-\Delta_p)^{s}$, which further reduces to the fractional Laplacian $(-\Delta)^s$ for $p=2$.
A prototype problem of type \eqref{maineqn} is
\begin{equation}\label{maineqnprot}
-\Delta_p u+(-\Delta_p)^s u=0, 
\quad 
1<p<\infty, 
\quad 
0<s<1.
\end{equation}

Before describing our contribution, let us discuss some of the known results. 
In the local case, the regularity theory for the $p$-Laplace equation $\Delta_p u=0$ has been studied extensively, for example, see Lindqvist \cite{Lind}, Mal\'{y}-Ziemer \cite{Maly} and references therein. 
For the nonlocal $p$-Laplace equation
 \begin{equation}\label{fracp}
 (-\Delta_p)^s u=0,
 \end{equation}
a scale invariant Harnack inequality holds for globally nonnegative solutions. However, such inequality fails when the solution changes sign as shown by Kassmann \cite{KassmanHarnack} in the case $p=2$, see also Dipierro-Savin-Valdinoci \cite{Valjems}. These results have been extended to $1<p<\infty$ by Di Castro-Kuusi-Palatucci \cite{KuusiHarnack}. 
In addition, a weak Harnack inequality for supersolutions of \eqref{fracp} has been discussed in \cite{KuusiHarnack}. 
They introduced a nonlocal tail term to compensate for the sign change in the Harnack estimates. 
Di Castro-Kuusi-Palatucci \cite{Kuusilocal} has studied local boundedness estimate along with H\"older continuity of solutions for \eqref{fracp}. See also Brasco-Lindgren-Scikorra \cite{Brascolind} for higher regularity results. For lower semicontinuity results of supersolutions, we refer to Korvenp\"{a}\"{a}-Kuusi-Lindgren \cite{lscfrac}.

For the mixed local and nonlocal case with $p=2$, i.e.
\begin{equation}\label{maineqnclas}
-\Delta u+(-\Delta)^s u=0,
\end{equation}
 Foondun \cite{Fo} has proved Harnack inequality and local H\"older continuity for nonnegative solutions. 
 Barlow-Bass-Chen-Kassmann \cite{BBCK} has obtained a Harnack inequality for the parabolic problem related to \eqref{maineqnclas}.
 Chen-Kumagai \cite{CK} has proved Harnack inequality and local H\"older continuity for the parabolic problem of \eqref{maineqnclas}.
 Such a parabolic Harnack estimate has been used to prove elliptic Harnack inequality for \eqref{maineqnclas} by Chen-Kim-Song-Vondra\v{c}ek in \cite{CKSV}.
 For more regularity results related to \eqref{maineqnclas}, we refer to Athreya-Ramachandran \cite{AR}, Chen-Kim-Song \cite{CKSheatest} and Chen-Kim-Song-Vondra\v{c}ek \cite{CKSVGreenest}. 
 The arguments in these articles combine probability and analysis. 
 Moreover, the Harnack inequality is proved only for globally nonnegative solutions.
Recently, an interior Sobolev regularity, a strong maximum principle and a symmetry property among many other qualitative properties of solutions to \eqref{maineqnclas} has been studied by Biagi-Dipierro-Valdinoci-Vecchi \cite{BSVV2, BSVV1}, Dipierro-Proietti Lippi-Valdinoci \cite{DPV20, DPV21} and Dipierro-Ros-Oton-Serra-Valdinoci \cite{DRXJV20}. 
There also exist regularity results for a nonhomogeneous analogue of  \eqref{maineqnclas}. More precisely, Athreya-Ramachandran \cite{AR} has proved Harnack inequality by probabilistic and analytic methods and authors in \cite{BSVV2, DPV20} has obtained boundedness, interior as well as boundary regularity results by analytic techniques. 
Biagi-Dipierro-Valdinoci-Vecchi \cite{En21} has obtained interior regularity results for a nonhomogeneous version of \eqref{maineqnprot}.
 
We establish the following regularity results for weak solutions (Definition \ref{subsupsolution}) of \eqref{maineqn} with $1<p<\infty$ and $0<s<1$.
 \begin{itemize}
 \item Local boundedness of weak subsolutions (Theorem \ref{thm1}). The argument is based on an energy estimate (Lemma \ref{energyest}), the Sobolev inequality and an iteration technique (Lemma \ref{iteration}). 
 \item Local H\"older continuity of weak solutions (Theorem \ref{Holder}). Local H\"older continuity is not a direct consequence of the Harnack inequality in the nonlocal case, see \cite{BBCK, Kassweakharnack}. 
We follow the approach of  Di Castro-Kuusi-Palatucci \cite{Kuusilocal} in which the local boundedness estimate and the logarithmic energy estimate (Lemma \ref{loglemma}) play an important role.
 \item Harnack inequality (Theorem \ref{thm2}) for weak solutions and weak Harnack inequality for weak supersolutions (Theorem \ref{thm3}). The expansion of positivity (Lemma \ref{DGLemma}), the local boundedness result and a tail estimate (Lemma \ref{tailest}) are crucial here.
\item Lower and upper semicontinuity of weak supersolutions and subsolutions, respectively (Theorem \ref{lscthm1} and Corollary \ref{uscthm}). This result is an adaptation to the mixed local and nonlocal case of a measure theoretic approach (Lemma \ref{lscthm}) in Liao \cite{Liao}. We refer to Banerjee-Garain-Kinnunen \cite{BGK} for an adaptation of this approach to a class of doubly nonlinear parabolic nonlocal problems.
 \end{itemize}
 
In contrast to the techniques from probability and analysis introduced in \cite{AR, BBCK, CKSheatest, CKSVGreenest, CKSV, CK, Fo}, our approach is purely analytic and based on the De Giorgi-Nash-Moser theory. 
To the best of our knowledge, all of our main results are new for $p\ne2$. Moreover, some of our main results (Theorem \ref{thm1}, Theorem \ref{thm3}, Theorem \ref{lscthm1} and Corollary \ref{uscthm}) seem to be new even for $p=2$. 
Furthermore, our approach applies to sign changing solutions.
In this respect, our Harnack estimate (Theorem \ref{thm2}) also extends the result of Chen-Kim-Song-Vondra\v{c}ek \cite{CKSV} and Foondun \cite{Fo} to sign changing solutions.
We introduce a tail term (Definition \ref{def.tail}), that differs from the one discussed in \cite{KuusiHarnack}, and a tail estimate (Lemma \ref{tailest}) 
that capture both local and nonlocal features of \eqref{maineqn}.
Technical novelties include an adaptation of the expansion of positivity technique (Lemma \ref{DGLemma}) for the mixed problem.

This article is organized as follows. In Section 2, we discuss some definitions and preliminary results. 
Necessary energy estimates are  proved in Section 3. In Sections 4 and 5, we establish the local boundedness and H\"older continuity results. 
In Sections 6 and 7, we obtain a tail estimate and the expansion of positivity property.
In Section 8, we prove Harnack and weak Harnack estimates. Finally, in Section 9, we establish the lower and upper semicontinuity results.
\section{Preliminaries}
In this section, we present some known results for fractional Sobolev spaces, see Di Nezza-Palatucci-Valdinoci \cite{Hitchhiker'sguide} for more details.

\begin{Definition}
Let $1<p<\infty$ and $0<s<1$. Assume that $\Omega$ is a domain in $\mathbb R^n$. 
The fractional Sobolev space $W^{s,p}(\Omega)$ is defined by
$$
W^{s,p}(\Omega)=\bigg\{u\in L^p(\Omega):\int_{\Omega}\int_{\Omega}\frac{|u(x)-u(y)|^p}{|x-y|^{n+ps}}\,dx\,dy<\infty\bigg\}
$$
and endowed with the norm
$$
\|u\|_{W^{s,p}(\Omega)}=\bigg(\int_{\Omega}|u(x)|^p\,dx+\int_{\Omega}\int_{\Omega}\frac{|u(x)-u(y)|^p}{|x-y|^{n+ps}}\,dx\,dy\bigg)^\frac{1}{p}.
$$
The fractional Sobolev space with zero boundary value $W_{0}^{s,p}(\Omega)$ consists of functions $u\in W^{s,p}(\mathbb{R}^n)$ with 
$u=0$ on $\mathbb{R}^n\setminus\Omega$.
\end{Definition}

Both $W^{s,p}(\Omega)$ and $W_{0}^{s,p}(\Omega)$ are reflexive Banach spaces, see \cite{Hitchhiker'sguide}. 
The space $W^{s,p}_{\loc}(\Omega)$ is defined by requiring that a function belongs to $W^{s,p}(\Omega')$ for every $\Omega'\Subset\Omega$,
where $\Omega'\Subset\Omega$ denotes that $\overline{\Omega'}$ is a compact subset of $\Omega$.
Throughout, we write $c$ or $C$ to denote a positive constant which may vary from line to line or even in the same line. 
The dependencies on parameters are written in the parentheses.

The next result asserts that the standard Sobolev space is continuously embedded in the fractional Sobolev space, see \cite[Proposition 2.2]{Hitchhiker'sguide}.
The argument applies a smoothness property of $\Omega$ so that we can extend functions from $W^{1,p}(\Omega)$ to $W^{1,p}(\R^n)$ and that the extension operator is bounded.

\begin{Lemma}\label{locnon}
Let $\Omega$ be a smooth bounded domain in $\mathbb{R}^n$, $1<p<\infty$ and $0<s<1$. 
There exists a positive constant $C=C(n,p,s)$ such that
$
||u||_{W^{s,p}(\Omega)}\leq C||u||_{W^{1,p}(\Omega)}
$
for every $u\in W^{1,p}(\Omega)$.
\end{Lemma}

The following result for the fractional Sobolev spaces with zero boundary value follows from \cite[Lemma 2.1]{Silvaarxiv}.
The main difference compared to Lemma \ref{locnon} is that the result holds for any bounded domain, 
since for the Sobolev spaces with zero boundary value, we always have zero extension to the complement.

\begin{Lemma}\label{locnon1}
Let $\Omega$ be a bounded domain in $\mathbb{R}^n$, $1<p<\infty$ and $0<s<1$. 
There exists a positive constant $C=C(n,p,s,\Omega)$ such that
\[
\int_{\mathbb{R}^n}\int_{\mathbb{R}^n}\frac{|u(x)-u(y)|^p}{|x-y|^{n+ps}}\,dx\, dy
\leq C\int_{\Omega}|\nabla u(x)|^p\,dx
\]
for every $u\in W_0^{1,p}(\Omega)$.
Here we consider the zero extension of $u$ to the complement of $\Omega$.
\end{Lemma}

The following version of the Gagliardo-Nirenberg-Sobolev inequality will be useful for us,
see \cite[Corollary 1.57]{Maly}.
\begin{Lemma}\label{c.omega_sobo}
Let $1<p<\infty$, $\Omega$ be an open set in $\R^n$ with $\lvert\Omega\rvert<\infty$ and
\begin{equation}\label{kappa}
\kappa=
\begin{cases}
\frac{n}{n-p},&\text{if}\quad 1<p<n,\\
2,&\text{if}\quad p\geq n.
\end{cases}
\end{equation}
There exists a positive constant $C=C(n,p)$ such that
\begin{equation}\label{e.friedrich}
\biggl(\int_\Omega \lvert u(x)\rvert^{\kappa p}\,dx\biggr)^{\frac{1}{\kappa p}}
\le C \lvert\Omega\rvert^{\frac{1}{n}-\frac{1}{p}+\frac{1}{\kappa p}} \biggl(\int_\Omega \lvert \nabla u(x)\rvert^{p}\,dx\biggr)^{\frac{1}{p}}
\end{equation}
for every $u\in W_0^{1,p}(\Omega)$.
\end{Lemma}

Next, we define the notion of weak solution to \eqref{maineqn}.
 
\begin{Definition}\label{subsupsolution}
A function $u\in L^{\infty}(\mathbb{R}^n)$ is a weak subsolution of \eqref{maineqn} if $u\in W_{\loc}^{1,p}(\Omega)$ and for every $\Omega'\Subset\Omega$ and nonnegative test functions $\phi\in W_{0}^{1,p}(\Omega')$, we have
\begin{equation}\label{weaksubsupsoln}
\begin{gathered}
\int_{\Omega'}|\nabla u|^{p-2}\nabla u\cdot\nabla\phi\,dx+\int_{\mathbb{R}^n}\int_{\mathbb{R}^n}\mathcal{A}(u(x,y)){(\phi(x)-\phi(y))}\,d\mu\leq0,
\end{gathered}
\end{equation}
where
\[
\mathcal{A}(u(x,y))=|u(x)-u(y)|^{p-2}(u(x)-u(y))
\quad\text{and}\quad
d\mu=K(x,y)\,dx\,dy.
\]
Analogously, a function $u$  is a weak supersolution of  \eqref{maineqn} if the integral in \eqref{weaksubsupsoln} is nonnegative for every nonnegative test functions $\phi\in W_{0}^{1,p}(\Omega')$.
A function $u$ is a weak solution of \eqref{maineqn} if the equality holds in \eqref{weaksubsupsoln} for every $\phi\in W_{0}^{1,p}(\Omega')$ without a sign restriction.
\end{Definition}

\begin{Remark}\label{rkreg}
The boundedness assumption, together with Lemma \ref{locnon} and Lemma \ref{locnon1}, ensures that  Definition \ref{subsupsolution} is well stated and
 the tail that will be defined in \eqref{loctail} is finite. 
 Under the assumption that the tail in \eqref{loctail} is bounded, our main results Theorem \ref{thm1}, Theorem \ref{Holder}, Theorem \ref{thm2} and Theorem \ref{thm3} hold true without the a priori boundedness assumption on the function. 
In such a case, the local boundedness follows from Theorem \ref{thm1}.
\end{Remark}

It follows directly from Definition \ref{subsupsolution} that $u$ is a weak subsolution of  \eqref{maineqn} if and only if $-u$ is a weak supersolution of \eqref{maineqn}. 
Moreover, for any $c\in\mathbb{R}$, $u+c$ is a weak solution of \eqref{maineqn} if and only if $u$ is a weak solution of \eqref{maineqn}. 
We discuss some further structural properties of weak solutions below.
We denote the positive and negative parts of $a\in\R$ by $a_+=\max\{a,0\}$ and $a_-=\max\{-a,0\}$, respectively. Also, the barred integral sign denotes the corresponding integral average. 

\begin{Lemma}\label{Solsubsup}
A function $u$ is a weak solution of  \eqref{maineqn} if and only if $u$ is a weak subsolution and a weak supersolution of \eqref{maineqn}.
\end{Lemma}

\begin{proof}
It follows immediately from Definition \ref{subsupsolution}, that a weak solution $u$ of \eqref{maineqn} is a weak subsolution and a weak supersolution of \eqref{maineqn}.
Conversely, assume that $u$ is both weak subsolution and weak supersolution of \eqref{maineqn}. 
Let $\Omega'\Subset\Omega$ and $\phi\in W_0^{1,p}(\Omega')$. Then $\phi_{+}$ and $\phi_{-}$ belong to $W_0^{1,p}(\Omega')$. Since $u$ is a weak subsolution, we have 
\begin{equation}\label{Sub}
\int_{\Omega'}|\nabla u|^{p-2}\nabla u\cdot\nabla\phi_{+}\,dx+\int_{\mathbb{R}^n}\int_{\mathbb{R}^n}\mathcal{A}(u(x,y))(\phi_+(x)-\phi_{+}(y))\,d\mu\leq 0.
\end{equation}
Analogously, since $u$ is a weak supersolution, we have
\begin{equation}\label{Sup}
\int_{\Omega'}|\nabla u|^{p-2}\nabla u\cdot\nabla\phi_{-}\,dx+\int_{\mathbb{R}^n}\int_{\mathbb{R}^n}\mathcal{A}(u(x,y))(\phi_{-}(x)-\phi_{-}(y))\,d\mu\geq 0.
\end{equation}
Subtracting \eqref{Sub} and \eqref{Sup} and using $\phi=\phi_{+}-\phi_{-}$, we obtain
\begin{equation*}
\int_{\Omega'}|\nabla u|^{p-2}\nabla u\cdot\nabla\phi\,dx+\int_{\mathbb{R}^n}\int_{\mathbb{R}^n}\mathcal{A}(u(x,y))(\phi(x)-\phi(y))\,d\mu\leq 0.
\end{equation*}
The reverse inequality holds by replacing $\phi$ with $-\phi$. Hence, $u$ is a weak solution of  \eqref{maineqn}.
\end{proof}

Next, we show that the property of being a weak subsolution is preserved under taking the positive part. Then, it follows immediately that $u_-$ is a weak subsolution of \eqref{maineqn}, whenever $u$ is a weak supersolution of \eqref{maineqn}. 

\begin{Lemma}\label{cutoffsub}
Assume that $u$ is a weak subsolution of \eqref{maineqn}. 
Then $u_+$ is  a weak subsolution of \eqref{maineqn}. 
\end{Lemma}

\begin{proof}
Consider functions $u_k=\min\{ku_{+},1\}$, $k=1,2,\dots$.
Then $(u_k)_{k=1}^\infty$ is an increasing sequence of functions in $W^{1,p}_{\loc}(\Omega)$ and $0\leq u_k\leq 1$ for every $k\in\mathbb{N}$. 
Let $\phi\in C_c^{\infty}(\Omega')$ be a nonnegative function. By choosing $u_k\phi\in W_0^{1,p}(\Omega')$ as a test function in \eqref{weaksubsupsoln}, we obtain
\begin{equation}\label{cutofftest}
0\geq\int_{\Omega'}|\nabla u|^{p-2}\nabla u\cdot\nabla(u_k\phi)\,dx+\int_{\mathbb{R}^n}\int_{\mathbb{R}^n}\mathcal{A}(u(x,y))(u_k(x)\phi(x)-u_k(y)\phi(y))\,d\mu
=I_1+I_2.
\end{equation}
\textbf{Estimate of $I_1$:} We observe that
\begin{equation}\label{cutoffI1}
I_1=\int_{\Omega'}|\nabla u|^{p-2}\nabla u\cdot\nabla(u_k\phi)\,dx=k\int_{\Omega'\cap\{0<u<\frac{1}{k}\}}\phi|\nabla u|^p\,dx+\int_{\Omega'}u_k|\nabla u|^{p-2}\nabla u\cdot\nabla\phi\,dx.
\end{equation}
\textbf{Estimate of $I_2$:} Let $x,y\in\mathbb{R}^n$. First, we consider the case when $u(x)>u(y)$.\\
If $u_k(x)=0$, then $u_k(y)=0$. Hence, we have
\begin{equation}\label{cutoffI2case1}
(u(x)-u(y))^{p-1}(u_k(x)\phi(x)-u_k(y)\phi(y))=0.
\end{equation}
If $u_k(y)>0$, then $u(y)=u_+(y)$. 
Under the assumption $u(x)>u(y)$, it follows that $u(x)=u_{+}(x)$ and $u_k(x)>u_{k}(y)$.
This implies that
\begin{equation}\label{cutoffI2case2}
\begin{split}
(u(x)-u(y))^{p-1}(u_k(x)\phi(x)-u_k(y)\phi(y))
&=(u_+(x)-u_{+}(y))^{p-1}(u_k(x)\phi(x)-u_k(y)\phi(y))\\
&\geq (u_+(x)-u_{+}(y))^{p-1}u_k(x)(\phi(x)-\phi(y)).
\end{split}
\end{equation}
If $u_k(y)=0$ and $u_k(x)>0$, then $u(x)>0\geq u(y)$ and hence
\begin{equation}\label{cutoffI2case3}
\begin{split}
(u(x)-u(y))^{p-1}(u_k(x)\phi(x)-u_k(y)\phi(y))
&=(u(x)-u(y))^{p-1}u_k(x)\phi(x)\\
&\geq(u_+(x)-u_{+}(y))^{p-1}u_k(x)\phi(x)\\
&\geq(u_+(x)-u_{+}(y))^{p-1}u_k(x)(\phi(x)-\phi(y)).
\end{split}
\end{equation}
Therefore, from \eqref{cutoffI2case1}, \eqref{cutoffI2case2} and \eqref{cutoffI2case3} we have
\begin{equation}\label{cutoffI2casexgeqy}
\mathcal{A}(u(x,y))(u_k(x)\phi(x)-u_k(y)\phi(y))
\geq(u_+(x)-u_+(y))^{p-1}u_k(x)(\phi(x)-\phi(y)).
\end{equation}
When $u(x)=u(y)$, the estimate \eqref{cutoffI2casexgeqy} hods true. In case of $u(x)<u(y)$, by interchanging the roles of $x$ and $y$ in the above estimates, we arrive at
\begin{equation}\label{cutoffI2casexleqy}
\mathcal{A}(u(x,y))(u_k(x)\phi(x)-u_k(y)\phi(y))
\geq(u_+(y)-u_+(x))^{p-1}u_k(y)(\phi(y)-\phi(x)).
\end{equation}
Combining the estimates \eqref{cutoffI1}, \eqref{cutoffI2casexgeqy} and \eqref{cutoffI2casexleqy} in \eqref{cutofftest} and letting $k\to\infty$, along with an application of the Lebesgue dominated convergence theorem, we obtain
\begin{equation}\label{cutoffest}
\int_{\Omega'}|\nabla u_+|^{p-2}\nabla u_+\cdot\nabla\phi\,dx+\int_{\mathbb{R}^n}\int_{\mathbb{R}^n}\mathcal{A}(u_+(x,y))(\phi(x)-\phi(y))\,d\mu\leq 0.
\end{equation}
By a density argument  \eqref{cutoffest} holds for every $\phi\in W_0^{1,p}(\Omega')$.
This shows that $u_+$ is a weak subsolution of \eqref{maineqn}.
\end{proof}

\section{Energy estimates}

The following energy estimate will be crucial for us. 
We denote an open ball with center $x_0\in\R^n$ and radius $r>0$ by $B_r(x_0)$.

\begin{Lemma}\label{energyest}
Let $u$ be a weak subsolution of \eqref{maineqn} and denote $w=(u-k)_{+}$ with $k\in\mathbb{R}$. 
There exists a positive constant $C=C(p,\Lambda)$ such that 
\begin{equation}\label{energyesteqn}
\begin{split}
&\int_{B_r(x_0)}\psi^p|\nabla w|^p\,dx+\int_{B_r(x_0)}\int_{B_r(x_0)}|w(x)\psi(x)-w(y)\psi(y)|^p \,d\mu\\
&\leq C\bigg(\int_{B_r(x_0)}w^p |\nabla\psi|^p\,dx+\int_{B_r(x_0)}\int_{B_r(x_0)}{\max\{w(x),w(y)\}^p|\psi(x)-\psi(y)|^p}\,d\mu\\
&\qquad+\esssup_{x\in\supp\psi}\int_{{\mathbb{R}^n\setminus B_r(x_0)}}{\frac{w(y)^{p-1}}{|x-y|^{n+ps}}}\,dy
\cdot\int_{B_r(x_0)}w\psi^p\,dx\bigg),
\end{split}
\end{equation}
whenever $B_r(x_0)\subset\Omega$ and $\psi\in C_c^{\infty}(B_r{(x_0)})$ is a nonnegative function. 
If $u$ is a weak supersolution of \eqref{maineqn}, the estimate in \eqref{energyesteqn} holds with $w=(u-k)_{-}$.
\end{Lemma}

\begin{proof}
Let $u$ be a weak subsolution of \eqref{maineqn}. 
For $w=(u-k)_+$, by choosing $\phi=w\psi^p$ as a test function in \eqref{weaksubsupsoln}, we obtain
\begin{equation}\label{energytest}
\begin{split}
0&\geq\int_{B_r(x_0)}|\nabla u|^{p-2}\nabla u\cdot\nabla(w\psi^p)\,dx+\int_{\mathbb{R}^n}\int_{\mathbb{R}^n}\mathcal{A}(u(x,y))(w(x)\psi(x)^p-w(y)\psi(y)^p)\,d\mu\\
&=I+J.
\end{split}
\end{equation}
Proceeding as in the proof of \cite[Page 14, Proposition 3.1]{Verenacontinuity}, for some constants $c=c(p)>0$ and $C=C(p)>0$, we have
\begin{equation}\label{energyIest}
\begin{split}
I&=\int_{B_r(x_0)}|\nabla u|^{p-2}\nabla u\cdot\nabla(w\psi^p)\,dx\\
&\geq c\int_{B_r(x_0)}\psi^p|\nabla w|^p\,dx-C\int_{B_r(x_0)}w^p|\nabla\psi|^p\,dx.
\end{split}
\end{equation}
Moreover, from the lines of the proof of \cite[Pages 1285--1287, Theorem 1.4]{Kuusilocal}, for some constants $c=c(p,\Lambda)>0$ and $C=C(p,\Lambda)>0$, we have
\begin{equation}\label{energyJest}
\begin{split}
J&=\int_{\mathbb{R}^n}\int_{\mathbb{R}^n}\mathcal{A}(u(x,y))(w(x)\psi(x)^p-w(y)\psi(y)^p)\,d\mu\\
&\geq c\int_{B_r(x_0)}\int_{B_r(x_0)}|w(x)\psi(x)-w(y)\psi(y)|^p\,d\mu\\
&\qquad-C\int_{B_r(x_0)}\int_{B_r(x_0)}\max\{w(x),w(y)\}^p|\psi(x)-\psi(y)|^p\,d\mu\\
&\qquad-C\esssup_{x\in\supp\psi}\int_{{\mathbb{R}^n\setminus B_r(x_0)}}{\frac{w(y)^{p-1}}{|x-y|^{n+ps}}}\,dy
\cdot\int_{B_r(x_0)}w\psi^p\,dx.
\end{split}
\end{equation}
By applying \eqref{energyIest} and \eqref{energyJest} in \eqref{energytest}, we obtain \eqref{energyesteqn}. 
In the case of a weak supersolution, the estimate in \eqref{energyesteqn} follows by applying the obtained result to $-u$ .
\end{proof}

Next we define a tail which appears in estimates throughout the article.

\begin{Definition}\label{def.tail}
Let $u$ be a weak subsolution or a weak supersolution of \eqref{maineqn} as in Definition \ref{subsupsolution}.
The tail of $u$ with respect to a ball $B_r(x_0)$ is defined by
\begin{equation}\label{loctail}
\Tail(u;x_0,r)=\bigg(r^{p}\int_{\mathbb{R}^n\setminus B_r(x_0)}\frac{|u(y)|^{p-1}}{|y-x_0|^{n+ps}}\,dy\bigg)^\frac{1}{p-1}.
\end{equation}
\end{Definition}

We prove an energy estimate which will be crucial to obtain a reverse H\"older inequality for weak supersolutions of \eqref{maineqn}.

\begin{Lemma}\label{energyforrev}
Let $q\in(1,p)$ and $d>0$.
Assume that $u$ is a weak supersolution of \eqref{maineqn} such that $u\geq 0$ in $B_R(x_0)\subset\Omega$ and denote by $w=(u+d)^\frac{p-q}{p}$. 
There exists a positive constant $c=c(p,\Lambda)$ such that
\begin{equation}\label{energyrevest1}
\begin{split}
&\int_{B_r(x_0)}\psi^p|\nabla w|^p\,dx
\leq c\bigg(\frac{(p-q)^p}{(q-1)^\frac{p}{p-1}}\int_{B_r(x_0)}w^p|\nabla\psi|^p\,dx\\
&+\frac{(p-q)^p}{(q-1)^{p}}\int_{B_r(x_0)}\int_{B_r(x_0)}\max\{w(x),w(y)\}^p|\psi(x)-\psi(y)|^p\,d\mu\\
&+\frac{(p-q)^p}{(q-1)}\bigg(\esssup_{z\in\supp\psi}\int_{\mathbb{R}^n\setminus B_r(x_0)}K(z,y)\,dy
+d^{1-p}R^{-p}\Tail(u_{-};x_0,R)^{p-1}\bigg)\int_{B_r(x_0)}w^p\psi^p\,dx\bigg),
\end{split}
\end{equation}
whenever $B_r(x_0)\subset B_{\frac{3R}{4}}(x_0)$ and $\psi\in C_c^{\infty}(B_r(x_0))$ is a nonnegative function.
Here $\Tail(\cdot)$ is defined in \eqref{loctail}.
\end{Lemma}

\begin{proof}
Let $d>0$, $v=u+d$ and $q\in[1+\epsilon,p-\epsilon]$ for $\epsilon>0$ small enough. 
Then $v$ is a weak supersolution of \eqref{maineqn}. 
By choosing $\phi=v^{1-q}\psi^p$ as a test function in \eqref{weaksubsupsoln}, we obtain
\begin{equation}\label{revestfinal}
\begin{split}
0&\leq \int_{B_r(x_0)}|\nabla v|^{p-2}\nabla v\cdot\nabla( v^{1-q}\psi^p)\,dx\\
&\qquad+\int_{B_r(x_0)}\int_{B_r(x_0)}\mathcal{A}( v(x,y))( v(x)^{1-q}\psi(x)^p- v(y)^{1-q}\psi(y)^p)\,d\mu\\
&\qquad+2\int_{\mathbb{R}^n\setminus B_r(x_0)}\int_{B_r(x_0)}\mathcal{A}( v(x,y)) v(x)^{1-q}\psi(x)^p\,d\mu\\
&=I_1+I_2+2I_3.
\end{split}
\end{equation}
\textbf{Estimate of $I_1$:} We observe that
\begin{equation}\label{revI1est}
\begin{split}
I_1&=\int_{B_r(x_0)}|\nabla v|^{p-2}\nabla v\cdot\nabla( v^{1-q}\psi^p)\,dx\\
&\leq(1-q)\int_{B_r(x_0)} v^{-q}|\nabla v|^{p}\psi^p\,dx+p\int_{B_r(x_0)} v^{1-q}|\nabla\psi||\nabla  v|^{p-1}\psi^{p-1}\,dx\\
&=(1-q)J_1+J_2,
\end{split}
\end{equation}
where
$$
J_1=\int_{B_r(x_0)} v^{-q}|\nabla v|^{p}\psi^p\,dx
$$
and
$$
J_2=p\int_{B_r(x_0)} v^{1-q}|\nabla\psi||\nabla  v|^{p-1}\psi^{p-1}\,dx.
$$
\textbf{Estimate of $J_2$:} By Young's inequality, we obtain
\begin{equation}\label{revJ2est}
J_2=p\int_{B_r(x_0)} v^{1-q}|\nabla\psi||\nabla  v|^{p-1}\psi^{p-1}\,dx\leq\frac{q-1}{2}J_1+\frac{c(p)}{(q-1)^\frac{1}{p-1}}\int_{B_r(x_0)}|\nabla\psi|^p v^{p-q}\,dx.
\end{equation}
By applying \eqref{revJ2est} in \eqref{revI1est}, for some constant $c=c(p)>0$, we have
\begin{equation}\label{revI1estfinal}
\begin{split}
I_1&\leq\frac{1-q}{2}\int_{B_r(x_0)} v^{-q}|\nabla v|^{p}\psi^p\,dx+\frac{c}{(q-1)^\frac{1}{p-1}}\int_{B_r(x_0)}|\nabla\psi|^p v^{p-q}\,dx\\
&=-\frac{q-1}{2}\Big(\frac{p}{p-q}\Big)^p\int_{B_r(x_0)}\big|\nabla( v^\frac{p-q}{p})\big|^p\psi^p\,dx
+\frac{c}{(q-1)^\frac{1}{p-1}}\int_{B_r(x_0)}|\nabla\psi|^p v^{p-q}\,dx.
\end{split}
\end{equation}
\textbf{Estimate of $I_2$:} Following the lines of the proof of \cite[Pages 1830--1833, Lemma 5.1]{KuusiHarnack} for $w= v^\frac{p-q}{p}$, with some positive constants $c(p,q)$ and $c(p)$, we obtain
\begin{equation}\label{revI2est}
\begin{split}
I_2&=\int_{B_r(x_0)}\int_{B_r(x_0)}\mathcal{A}( v(x,y))( v(x)^{1-q}\psi(x)^p- v(y)^{1-q}\psi(y)^p)\,d\mu\\
&\leq -c(p,q)\int_{B_r(x_0)}\int_{B_r(x_0)}|w(x)-w(y)|^p\,\psi(y)^{p}\,d\mu\\
&\qquad+\frac{c(p)}{(q-1)^{p-1}}\int_{B_r(x_0)}\int_{B_r(x_0)}\max\{w(x),w(y)\}^p|\psi(x)-\psi(y)|^p\,d\mu.
\end{split}
\end{equation}
\\
\textbf{Estimate of $I_3$:} Following the lines of the proof of \cite[Page 1830, Lemma 5.1]{KuusiHarnack} for $w= v^\frac{p-q}{p}$, we obtain
\begin{equation}\label{revI34est}
\begin{split}
I_3&=2\int_{\mathbb{R}^n\setminus B_r(x_0)}\int_{B_r(x_0)}\mathcal{A}( v(x,y)) v(x)^{1-q}\psi(x)^p\,d\mu\\
&\leq c\bigg(\esssup_{z\in\supp\psi}\int_{\mathbb{R}^n\setminus B_r(x_0)}K(z,y)\,dy\\
&\qquad\qquad+d^{1-p}\int_{\mathbb{R}^n\setminus B_R(x_0)}(u(y))_{-}^{p-1}|y-x_0|^{-n-ps}\,dy\bigg)
\int_{B_r(x_0)}w^p\psi^p\,dx,
\end{split}
\end{equation}
with $c=c(p,\Lambda)>0$. By applying \eqref{revI1estfinal}, \eqref{revI2est} and \eqref{revI34est} in \eqref{revestfinal}, we obtain \eqref{energyrevest1}.
\end{proof}

Next, we obtain a logarithmic energy estimate.

\begin{Lemma}\label{loglemma}
Assume that $u$ is a weak supersolution of \eqref{maineqn} such that $u\geq 0$ in $B_R(x_0)\subset\Omega$. 
There exists a positive constant $c=c(n,p,s,\Lambda)$ such that
\begin{equation}\label{logest}
\begin{split}
&\int_{B_r(x_0)}|\nabla\log(u+d)|^p\,dx+\int_{B_r(x_0)}\int_{B_r(x_0)}\bigg|\log\bigg(\frac{u(x)+d}{u(y)+d}\bigg)\bigg|^p\,d\mu\\
&\qquad\leq cr^n\big(r^{-p}+r^{-ps}+d^{1-p}R^{-p}\Tail(u_-;x_0,R)^{p-1}\big),
\end{split}
\end{equation}
whenever $B_r(x_0)\subset B_\frac{R}{2}(x_0)$ and $d>0$.
Here $\Tail(\cdot)$ is given by \eqref{loctail}.
\end{Lemma}

\begin{proof}
Let $\psi\in C_c^{\infty}(B_\frac{3r}{2}(x_0))$ be such that $0\leq\psi\leq 1$ in $B_\frac{3r}{2}(x_0)$, $\psi= 1$ in $B_r(x_0)$,
and $|\nabla\psi|\leq\frac{8}{r}$ in $B_\frac{3r}{2}(x_0)$.
By choosing $\phi=(u+d)^{1-p}\psi^p$ as a test function in \eqref{weaksubsupsoln}, we obtain
\begin{equation}\label{logesttest}
\begin{split}
0&\leq\int_{B_{2r}(x_0)}\int_{B_{2r}(x_0)}\mathcal{A}(u(x,y))((u(x)+d)^{1-p}\psi(x)^p-(u(y)+d)^{1-p}\psi(y)^p)\,d\mu\\
&\qquad+2\int_{\mathbb{R}^n\setminus {B_{2r}}(x_0)}\int_{B_{2r}(x_0)}\mathcal{A}(u(x,y))(u(x)+d)^{1-p}\psi(x)^p\,d\mu\\
&\qquad+\int_{B_{2r}(x_0)}|\nabla u|^{p-2}\nabla u\cdot\nabla((u+d)^{1-p}\psi^p)\,dx\\
&=I_1+I_2+I_3.
\end{split}
\end{equation}
\textbf{Estimate of $I_1$:} Following the lines of the proof of \cite[Pages 1288--1289, Lemma 1.3]{Kuusilocal} and using the properties of $\psi$, for some positive constant $c=c(n,p,s,\Lambda)$, we obtain
\begin{equation}\label{logI1}
\begin{split}
I_1&=\int_{B_{2r}(x_0)}\int_{B_{2r}(x_0)}\mathcal{A}(u(x,y))((u(x)+d)^{1-p}\psi(x)^p-(u(y)+d)^{1-p}\psi(y)^p)\,d\mu\\
&\leq-\frac{1}{c}\int_{B_{2r}(x_0)}\int_{B_{2r}(x_0)}K(x,y)\bigg|\log\bigg(\frac{u(x)+d}{u(y)+d}\bigg)\bigg|^p\psi(y)^p\,dx\, dy+cr^{n-ps}.
\end{split}
\end{equation}
\textbf{Estimate of $I_2$:} Following the lines of the proof of \cite[Page 1290, Lemma 1.3]{Kuusilocal} and using the properties of $\psi$, for some positive constant $c=c(n,p,s,\Lambda)$, we obtain
\begin{equation}\label{logI2}
\begin{split}
I_2&=2\int_{\mathbb{R}^n\setminus {B_{2r}}(x_0)}\int_{B_{2r}(x_0)}\mathcal{A}(u(x,y))(u(x)+d)^{1-p}\psi(x)^p\,d\mu\\
&\leq c d^{1-p}r^n R^{-p}\Tail(u_-;x_0,R)^{p-1}+c r^{n-ps}.
\end{split}
\end{equation}
\textbf{Estimate of $I_3$:} 
Arguing similarly as in the proof of \cite[Pages 717-718, Lemma 3.4]{Kin-Kuusi} and using the properties of $\psi$, for some positive constant $c=c(p)$, we have
\begin{equation}\label{logI3}
I_3=\int_{B_{2r}(x_0)}|\nabla u|^{p-2}\nabla u\cdot\nabla((u+d)^{1-p}\psi^p)\,dx\leq-c\int_{B_r(x_0)}|\nabla\log(u+d)|^p\,dx+c r^{n-p}.
\end{equation}
Hence using \eqref{logI1}, \eqref{logI2} and \eqref{logI3} in \eqref{logesttest} along with the fact that $\psi\equiv 1$ in $B_r(x_0)$, the estimate \eqref{logest} follows.
\end{proof}

As a consequence of Lemma \ref{loglemma}, we have the following result.

\begin{Corollary}\label{cor}
Assume that $u$ is a weak solution of \eqref{maineqn} such that $u\geq 0$ in $B_R(x_0)\subset\Omega$. Let $a,d>0$, $b>1$ and denote
$$
v=\min\bigg\{\bigg(\log\bigg(\frac{a+d}{u+d}\bigg)\bigg)_+,\log b\bigg\}.
$$
There exists a positive constant $c=c(n,p,s,\Lambda)$ such that
\begin{equation}\label{corest}
\fint_{B_r(x_0)}|v-(v)_{B_r(x_0)}|^p\,dx\leq c\Big(1+d^{1-p}\Big(\frac{r}{R}\Big)^p\Tail(u_-;x_0,R)^{p-1}\Big),
\end{equation}
whenever $B_r(x_0)\subset B_\frac{R}{2}(x_0)$ with $r\in(0,1]$.
Here $(v)_{B_r(x_0)}=\fint_{B_r(x_0)}v\,dx$ and $\Tail(\cdot)$ is given by \eqref{loctail}.
\end{Corollary}

\begin{proof}
By the Poincar\'e inequality {from \cite[Theorem 2]{Evans}}, for a constant $c=c(n,p)>0$, we have
\begin{equation}\label{Sobapp}
\fint_{B_r(x_0)}|v-(v)_{B_r(x_0)}|^p\,dx\leq c r^{p-n}\int_{B_r(x_0)}|\nabla v|^p\,dx.
\end{equation}
Now since $v$ is a truncation of the sum of a constant and $\log(u+d)$, we have
\begin{equation}\label{trunprp}
\int_{B_r(x_0)}|\nabla v|^p\,dx\leq \int_{B_r(x_0)}|\nabla\log(u+d)|^p\,dx.
\end{equation}
The estimate in \eqref{corest} follows by employing \eqref{logest} in \eqref{trunprp} along with \eqref{Sobapp} and the fact that $r\in(0,1]$.
\end{proof}

\section{Local boundedness}
We apply the following real analysis lemma. For the proof of Lemma \ref{iteration}, see \cite[Lemma 4.1]{Dibe}.
\begin{Lemma}\label{iteration}
Let $(Y_j)_{j=0}^{\infty}$ be a sequence of positive real numbers such that
$Y_0\leq c_{0}^{-\frac{1}{\beta}}b^{-\frac{1}{\beta^2}}$ and $Y_{j+1}\leq c_0 b^{j} Y_j^{1+\beta}$,
$j=0,1,2,\dots$, for some constants $c_0,b>1$ and $\beta>0$. Then $\lim_{j\to\infty}\,Y_j=0$.
\end{Lemma}

Our first main result shows that weak subsolutions of \eqref{maineqn} are locally bounded.
This result comes with a useful estimate.

\begin{Theorem}\label{thm1}(\textbf{Local boundedness}).
Let $u$ be a weak subsolution of \eqref{maineqn}. 
There exists a positive constant $c=c(n,p,s,\Lambda)$, such that
\begin{equation}\label{locbd}
\esssup_{B_{\frac{r}{2}}(x_0)}\,u
\leq \delta \Tail\big(u_{+};x_0,\tfrac{r}{2}\big)+c\delta^{-\frac{(p-1)\kappa}{p(\kappa-1)}}\bigg(\fint_{B_r(x_0)}u_{+}^p\,dx\bigg)^\frac{1}{p},
\end{equation}
whenever $B_r(x_0)\subset\Omega$ with $r\in(0,1]$ and $\delta\in(0,1]$. 
Here $\kappa$ and $\Tail(\cdot)$ are given by \eqref{kappa} and \eqref{loctail}, respectively.
\end{Theorem}
\begin{proof}
Let $B_r(x_0)\subset\Omega$ with $r\in(0,1]$. For $j=0,1,2,\dots$, we denote
$r_j=\frac{r}{2}(1+2^{-j})$, $\bar{r}_j=\frac{r_j+r_{j+1}}{2}$,
$B_j=B_{r_j}(x_0)$ and $\bar{B}_{j}=B_{\bar{r}_j}(x_0)$.
Let $(\psi_j)_{j=0}^{\infty}\subset C_c^{\infty}(\bar{B}_j)$ be a sequence of cutoff functions such that
$0\leq\psi_j\leq 1$ in $\bar{B}_{j}$, $\psi_j=1$ in $B_{j+1}$ and $|\nabla\psi_j|\leq\frac{2^{j+3}}{r}$ for every $j=0,1,2,\dots$.
For $j=0,1,2,\dots$ and $k,\bar{k}\geq 0$, we denote
$k_j=k+(1-2^{-j})\bar{k}$,
$\bar{k}_j=\frac{k_j+k_{j+1}}{2}$,
$w_j=(u-k_j)_{+}$ and $\bar{w}_j=(u-\bar{k}_j)_{+}$.
Then there exists a constant $c=c(n,p)>0$ such that
\begin{equation}\label{locest}
\begin{split}
\Big(\frac{\bar{k}}{2^{j+2}}\Big)^\frac{p(\kappa-1)}{\kappa}
\bigg(\fint_{B_{j+1}}w_{j+1}^p\,dx\bigg)^\frac{1}{\kappa}&=(k_{j+1}-\bar{k}_j)^\frac{p(\kappa-1)}{\kappa}
\bigg(\fint_{B_{j+1}}w_{j+1}^p\,dx\bigg)^\frac{1}{\kappa}\\
&\leq c\bigg(\fint_{\bar{B}_j}|\bar{w}_j\psi_j|^{p\kappa}\,dx\bigg)^\frac{1}{\kappa}, 
\end{split}
\end{equation}
where $\kappa$ is given by \eqref{kappa}. 
By the Sobolev inequality in \eqref{e.friedrich}, with $c=c(n,p,s)>0$, we obtain 
\begin{equation}\label{locbdSob1}
\begin{split}
\bigg(\fint_{\bar{B}_j}|\bar{w}_j\psi_j|^{p\kappa}\,dx\bigg)^\frac{1}{\kappa}
&\leq c r^{p-n}\int_{B_j}|\nabla(\bar{w}_j\psi_j)|^p\,dx\\
&\leq c r^{p-n}\bigg(\int_{B_j}\bar{w}_j^p |\nabla\psi_j|^p\,dx+\int_{B_j}\psi_j^{p}|\nabla\bar{w}_j|^p\,dx\bigg)=I_1+I_2.
\end{split}
\end{equation}
\textbf{Estimate of $I_1$:} Using the properties of $\psi_j$, for some $c=c(n,p,s)>0$, we have
\begin{equation}\label{locestI1}
I_1=cr^{p-n}\int_{B_j}\bar{w}_j^p |\nabla\psi_j|^p\,dx
\leq c2^{jp}\fint_{B_j} w_j^{p}\,dx.
\end{equation}
\textbf{Estimate of $I_2$:} By Lemma \ref{energyest}, with $c=c(n,p,s)$ and $C=(n,p,s,\Lambda)$ positive, we obtain
\begin{equation}\label{locestI2}
\begin{split}
I_2&=cr^{p-n}\int_{B_j}\psi_j^{p}|\nabla\bar{w}_j|^p\,dx\\
&\leq Cr^{p-n}\bigg(\int_{B_j}\bar{w}_j^p|\nabla\psi_j|^p\,dx
+\int_{B_j}\int_{B_j}\max\{\bar{w}_j(x),\bar{w}_j(y)\}^p|\psi_j(x)-\psi_j(y)|^p\,d\mu\\
&\qquad+\int_{B_j}\bar{w}_j(y)\psi_j(y)^p\,dy
\cdot\esssup_{y\in\supp\psi_j}\int_{\mathbb{R}^n\setminus B_j}{\bar{w}_j(x)}^{p-1}K(x,y)\,dx\bigg)\\
&=J_1+J_2+J_3.
\end{split}
\end{equation}
\textbf{Estimates of $J_1$ and $J_2$:} To estimate $J_1$, we use the estimate of $I_1$ in \eqref{locestI1} above and to estimate $J_2$, proceeding similarly as in the proof of the estimate $(4.5)$ in \cite[Page 1292]{Kuusilocal} and again using the properties of $\psi_j$, for every $r\in(0,1]$, we obtain  
\begin{equation}\label{locestJ12}
J_i\leq c(n,p,s,\Lambda)2^{jp}\fint_{B_j}w_j^p\,dx,\quad j=1,2.
\end{equation}
\textbf{Estimate of $J_3$:} We observe that $w_j^p\geq(\bar{k}_j-k_j)^{p-1}\bar{w}_j$. 
For any $\delta\in(0,1]$, we have 
\begin{equation}\label{locestJ3}
\begin{split}
J_3&=
C(n,p,s,\Lambda)r^{p-n}\int_{B_j}\bar{w}_j(y)\psi_j(y)^p\,dy
\cdot\esssup_{y\in\supp\psi_j}\int_{\mathbb{R}^n\setminus B_j}{\bar{w}_j(x)^{p-1}}K(x,y)\,dx\\
&\leq c2^{j(n+ps)}r^p\fint_{B_j}\frac{{w}_j(y)^p}{(\bar{k}_j-k_j)^{p-1}}\,dy\int_{\mathbb{R}^n\setminus B_j}\frac{w_j(x)^{p-1}}{|x-x_0|^{n+ps}}\,dx\\
&\leq c\frac{2^{j(n+ps+p-1)}}{\bar{k}^{p-1}}\Tail(w_0;x_0,\tfrac{r}{2})^{p-1}\fint_{B_j}w_j(y)^p\,dy\\
&\leq c2^{j(n+ps+p-1)}\delta^{1-p}\fint_{B_j}w_j(y)^p\,dy, 
\end{split}
\end{equation}
with $c=c(n,p,s,\Lambda)>0$, whenever $\bar{k}\geq\delta\Tail(w_0;x_0,\tfrac{r}{2})$.
Here we used the fact that
$$
\frac{|x-x_0|}{|x-y|}\leq\frac{|x-y|+|y-x_0|}{|x-y|}\leq 1+\frac{\bar{r}_j}{r_j-\bar{r}_j}\leq 2^{j+4},
$$
which holds for $x\in\mathbb{R}^n\setminus B_j$ and $y\in\supp\psi_j=\bar{B}_j$.

By applying \eqref{locestJ12} and \eqref{locestJ3} in \eqref{locestI2}, we obtain
\begin{equation}\label{locestI2final}
I_2\leq c(n,p,s,\Lambda)2^{j(n+ps+p-1)}\delta^{1-p}\fint_{B_j}w_j^p\,dx
\end{equation}
for every $\delta\in(0,1]$.
Inserting \eqref{locestI1} and \eqref{locestI2final} into \eqref{locbdSob1} we have
\begin{equation}\label{locbdSob12}
\begin{split}
\bigg(\fint_{\bar{B}_j}|\bar{w}_j\psi_j|^{p\kappa}\,dx\bigg)^\frac{1}{\kappa}&\leq c(n,p,s,\Lambda)2^{j(n+ps+p-1)}\delta^{1-p}\fint_{B_j}w_j^p\,dx.
\end{split}
\end{equation}
Setting
$$
Y_j=\bigg(\fint_{B_{j}}w_{j}^p\,dx\bigg)^\frac{1}{p},
$$
and
$$
\bar{k}=\delta\Tail(w_0;x_0,\tfrac{r}{2})+c_0^{\frac{1}{\beta}}b^{\frac{1}{\beta^2}}\bigg(\fint_{B_r(x_0)} w_{0}^p\,dx\bigg)^\frac{1}{p},
$$
where
$$
c_0=c(n,p,s,\Lambda)\delta^{\frac{(1-p)\kappa}{p}},
\quad 
b=2^{(\frac{n+ps+p-1}{p}+\frac{\kappa-1}{\kappa})\kappa}
\quad\text{and}\quad
\beta=\kappa-1.
$$
From \eqref{locest} and \eqref{locbdSob12} we obtain
\begin{equation}\label{lociteration}
\frac{Y_{j+1}}{\bar{k}}\leq c(n,p,s,\Lambda)2^{j(\frac{n+ps+p-1}{p}+\frac{\kappa-1}{\kappa})\kappa}\delta^{\frac{(1-p)\kappa}{p}}\Big(\frac{Y_j}{\bar{k}}\Big)^{\kappa}.
\end{equation}
Moreover, by the definition of $\bar{k}$ above we have
$$
\frac{Y_0}{\bar{k}}\leq c_0^{-\frac{1}{\beta}}b^{-\frac{1}{\beta^2}}.
$$
Thus from Lemma \ref{iteration}, we obtain $Y_j\to0$ as $j\to\infty$.
This implies that
$$
\esssup_{B_{\frac{r}{2}}(x_0)}\,u\leq k+\bar{k},
$$
which gives \eqref{locbd} by choosing $k=0$.
\end{proof}

\section{Oscillation estimates}
 
The following local H\"older continuity result for weak solutions of \eqref{maineqn} follows from Lemma \ref{osclemma} below.

\begin{Theorem}\label{Holder}(\textbf{H\"older continuity})
Let $u$ be a weak solution of  \eqref{maineqn}. Then $u$ is locally H\"older continuous in $\Omega$. 
Moreover, there exist constants $\alpha\in(0,\frac{p}{p-1})$ and $c=c(n,p,s,\Lambda)$, such that
\begin{equation}\label{Holderest}
\osc_{B_{\rho}(x_0)}\,u=\esssup_{B_{\rho}(x_0)}\,u-\essinf_{B_{\rho}(x_0)}\,u
\leq c\Big(\frac{\rho}{r}\Big)^\alpha\bigg(\Tail(u;x_0,r)+\bigg(\fint_{B_{2r}(x_0)}|u|^p\,dx\bigg)^\frac{1}{p}\bigg),
\end{equation}
whenever $B_{2r}(x_0)\subset\Omega$ with $r\in(0,1]$ and $\rho\in(0,r]$. 
Here $\Tail(\cdot)$ is given by \eqref{loctail}.
\end{Theorem}

We prove the next result by arguing similarly as in the proof of \cite[Lemma 5.1]{Kuusilocal}.

\begin{Lemma}\label{osclemma}
Let $u$ be a weak solution of \eqref{maineqn} and $0<r<\frac{R}{2}$ for some $R$ such that $B_{R}(x_0)\subset\Omega$ with $r\in(0,1]$. For $\eta\in(0,\frac{1}{4}]$, we set
$r_j=\eta^j\frac{r}{2}$ and $B_{j}=B_{r_{j}}(x_0)$ for $j=0,1,2,\dots$.
Denote
\begin{equation}\label{omm}
\frac{1}{2}\omega(r_0)=\Tail\big(u;x_0,\tfrac{r}{2}\big)+c\bigg(\fint_{B_r(x_0)}|u|^p\,dx\bigg)^\frac{1}{p},
\end{equation}
where $\Tail(\cdot)$ is given by \eqref{loctail}, $c=c(n,p,s,\Lambda)$ is the constant in \eqref{locbd} and let
\begin{equation}\label{om}
\omega(r_j)=\Big(\frac{r_j}{r_0}\Big)^{\alpha}\omega(r_0),
\quad j=1,2,\dots,
\end{equation}
for some $\alpha\in(0,\frac{p}{p-1})$.
Then
\begin{equation}\label{oscest}
\osc_{B_j}u\leq\omega(r_j),
\quad j=0,1,2,\dots.
\end{equation}
\end{Lemma}

\begin{proof}
Lemma \ref{Solsubsup} and Lemma \ref{cutoffsub} imply that $u_+$ and $(-u)_+$ are weak subsolutions of \eqref{maineqn}. 
By applying Theorem \ref{thm1} with $u_+$ and $(-u)_+$, we observe that \eqref{oscest} holds true for $j=0$.

Suppose \eqref{oscest} holds for every $i=0,\ldots,j$ for some $j\in\{0,1,2,\dots\}$. To obtain \eqref{oscest}, by induction, it is enough to deduce \eqref{oscest} for $i=j+1$. We prove it in two steps below. In Step $1$, we obtain the estimate \eqref{s1} below related to $u_j$, where $u_j$ will be defined in \eqref{uj}. In Step $2$, we use the estimate \eqref{s1} along with an iteration argument to conclude the proof of \eqref{oscest}.

We observe that either
\begin{equation}\label{p1}
\frac{\big|B_{2 r_{j+1}}(x_0)\cap\big\{u\geq\essinf_{B_j}+\frac{\omega(r_j)}{2}\big\}\big|}{|B_{2 r_{j+1}}(x_0)|}\geq\frac{1}{2}
\end{equation}
or
\begin{equation}\label{p2}
\frac{\big|B_{2 r_{j+1}}(x_0)\cap\big\{u\leq\essinf_{B_j}+\frac{\omega(r_j)}{2}\big\}\big|}{|B_{2 r_{j+1}}(x_0)|}\geq\frac{1}{2}
\end{equation}
holds. Let
\begin{equation}\label{uj}
u_j=
\begin{cases}
u-\essinf_{B_j}u,\quad \text{if }\eqref{p1} \text{ holds},\\
\omega(r_j)-(u-\essinf_{B_j}u),\quad\text{if }\eqref{p2} \text{ holds}.
\end{cases}
\end{equation}
Then $u_j$ is a weak solution of \eqref{maineqn}. Also, in both cases of \eqref{p1} and \eqref{p2}, $u_j\geq 0$ in $B_j$ and 
\begin{equation}\label{p12}
\frac{\big|B_{2 r_{j+1}}(x_0)\cap\big\{u_j\geq\frac{\omega(r_j)}{2}\big\}\big|}{|B_{2 r_{j+1}}(x_0)|}\geq\frac{1}{2}.
\end{equation}
\textbf{Step $1$:} We claim that
\begin{equation}\label{s1}
\frac{|B_{2 r_{j+1}}(x_0)\cap\{u_j\leq 2\varepsilon\omega(r_j)\}|}{|B_{2 r_{j+1}}(x_0)|}\leq\frac{\hat{C}}{\log(\frac{1}{\eta})},
\end{equation}
where $\varepsilon=\eta^{\frac{p}{p-1}-\alpha}$ for some positive constant $\hat{C}$ depending only on $n,p,s,\Lambda$ and the difference between $\frac{p}{p-1}$ and $\alpha$ via the definition of $\varepsilon$. To this end, we will apply the logarithmic estimate from Corollary \ref{cor}, where a tail quantity appears. We set
\begin{equation}\label{mu}
\mu=\log\bigg(\frac{\frac{\omega(r_j)}{2}+\varepsilon\omega(r_j)}{3\varepsilon\omega(r_j)}\bigg)
=\log\bigg(\frac{\frac{1}{2}+\varepsilon}{3\varepsilon}\bigg)\approx\log\Big(\frac{1}{\varepsilon}\Big)
\end{equation}
and define 
\begin{equation}\label{theta}
\Theta=\min\bigg\{\bigg(\log\bigg(\frac{\frac{\omega(r_j)}{2}+\varepsilon \omega(r_j)}{u_j+\varepsilon \omega(r_j)}\bigg)\bigg)_+,\mu\bigg\}.
\end{equation}
By \eqref{p12} we have
\begin{equation}\label{lset}
\begin{split}
\mu&=\frac{1}{|B_{2 r_{j+1}}(x_0)\cap\{u_j\geq\frac{\omega(r_j)}{2}\}|}\int_{B_{2 r_{j+1}}(x_0)\cap\{u_j\geq\frac{\omega(r_j)}{2}\}}\mu\,dx\\
&=\frac{1}{|B_{2 r_{j+1}}(x_0)\cap\{u_j\geq\frac{\omega(r_j)}{2}\}|}\int_{B_{2 r_{j+1}}(x_0)\cap\{\Theta=0\}}\mu\,dx\\
&\leq\frac{2}{|B_{2 r_{j+1}}(x_0)|}\int_{B_{2 r_{j+1}}(x_0)}(\mu-\Theta)\,dx=2(\mu-(\Theta)_{B_{2 r_{j+1}}(x_0)}),
\end{split}
\end{equation}
where $(\Theta)_{B_{2 r_{j+1}}(x_0)}=\fint_{B_{2 r_{j+1}}(x_0)}\Theta\,dx$. 
Integrating \eqref{lset} over the set $\{B_{2 r_{j+1}}(x_0)\cap \Theta=\mu\}$ we get
\begin{equation}\label{lsetj}
\begin{split}
\frac{|B_{2 r_{j+1}}(x_0)\cap\{\Theta=\mu\}|}{|B_{2 r_{j+1}}(x_0)|}\mu&\leq\frac{2}{|B_{2 r_{j+1}}(x_0)|}\int_{B_{2 r_{j+1}}(x_0)}|\Theta-(\Theta)_{B_{2 r_{j+1}}(x_0)}|\,dx.
\end{split}
\end{equation}
Applying Corollary \ref{cor} with $a=\frac{\omega(r_j)}{2}$, $d=\varepsilon\omega(r_j)$ and $b=e^\mu$ for some constant $c=c(n,p,s,\Lambda)$ we obtain
\begin{equation}\label{c1app}
\begin{split}
\fint_{B_{2 r_{j+1}}(x_0)}|\Theta-(\Theta)_{B_{2 r_{j+1}}(x_0)}|^p\,dx
\leq c\bigg((\varepsilon\omega(r_j))^{1-p}\Big(\frac{r_{j+1}}{r_j}\Big)^p\mathrm{Tail}(u_j;x_0,r_j)^{p-1}+1\bigg).
\end{split}
\end{equation}
Noting that $\eta\in(0,\frac{1}{4}]$, $\alpha\in(0,\frac{p}{p-1})$ along with $r\in(0,1]$ and following the lines of the proof of the estimate $(5.6)$ in \cite[Pages 1294--1295]{Kuusilocal}, we obtain
\begin{equation}\label{tailest1}
\mathrm{Tail}(u_j;x_0,r_j)^{p-1}\leq c\eta^{-\alpha(p-1)}\omega(r_j)^{p-1}
\end{equation}
for some positive constant $c$ depending only on $n,p,s$ and the difference between $\frac{p}{p-1}$ and $\alpha$, but independent of $\eta$. Therefore, using the estimate \eqref{tailest1} in \eqref{c1app} we obtain
\begin{equation}\label{c1app12}
\begin{split}
\fint_{B_{2 r_{j+1}}(x_0)}|\Theta-(\Theta)_{B_{2 r_{j+1}}(x_0)}|\,dx\leq C
\end{split}
\end{equation}
for some positive constant $C$ depending only on $n,p,s,\Lambda$ and the difference between $\frac{p}{p-1}$ and $\alpha$. 
The estimate \eqref{s1} follows by employing \eqref{c1app12} in \eqref{lsetj}.\\
\textbf{Step $2$:} Now we use an iteration argument to obtain \eqref{oscest} for $i=j+1$. To this end, for every $i=0,1,2,\dots$, let
$\rho_i=(1+2^{-i})r_{j+1}$, $\hat{\rho}_i=\frac{\rho_i +\rho_{i+1}}{2}$, $B^i=B_{\rho_i}(x_0)$ and $\hat{B}^i=B_{\hat{\rho}_i}(x_0)$.
Recalling that $\varepsilon=\eta^{\frac{p}{p-1}-\alpha}$, we denote
 $k_i=(1+2^{-i})\varepsilon\omega(r_j)$ and
$$
A^i=B^i\cap\{u_j\leq k_i\},\quad i=0,1,2,\dots.
$$
Let $w_i=(k_i -u_j)_+$ and $(\psi_i)_{i=0}^{\infty}\subset C_c^{\infty}(\hat{B}^i)$ be such that $0\leq \psi_i\leq 1$ in $\hat{B}^i$, $\psi_i =1$ in $B^{i+1}$ and $|\nabla\psi_i|\leq\frac{c 2^i}{\rho_i}$ in $\hat{B}^i$, with $c=c(n,p)>0$. By applying the Sobolev inequality in \eqref{e.friedrich}, for $\kappa$ as defined in \eqref{kappa}, we obtain a constant $c=c(n,p,s)>0$ such that
\begin{equation}\label{Poincareless1}
\begin{split}
(k_i-k_{i+1})^p\bigg(\frac{|A^{i+1}|}{|B^{i+1}|}\bigg)^\frac{1}{\kappa}&\leq\bigg(\fint_{B^{i+1}}w_i^{\kappa p}\,dx\bigg)^\frac{1}{\kappa}\leq c \bigg(\fint_{B^{i}}w_i^{\kappa p}\psi_i^{\kappa p}\,dx\bigg)^\frac{1}{\kappa}\\
&\leq c r_{j+1}^p\fint_{B^i}|\nabla(w_i \psi_i)|^p\,dx\leq cr_{j+1}^p(I+J),
\end{split}
\end{equation}
where 
$$
I=\fint_{B^i}w_i^{p}|\nabla\psi_i|^p\,dx
\quad\text{and}\quad
J=\fint_{B^i}|\nabla w_i|^{p}\psi_{i}^p\,dx.
$$
\textbf{Estimate of $I$:} Since $u_j\geq 0$ in $B_{j}$, we have $w_i\leq k_i\leq 2\varepsilon\omega(r_j)$ in $B^i$. Thus, using the properties of $\psi_i$ above, for some constant $c=c(n,p)>0$, we have
\begin{equation}\label{Jite}
\begin{split}
I&=\fint_{B^i}w_i^{p}|\nabla\psi_i|^p\,dx\leq cr_{j+1}^{-p}(\varepsilon\omega(r_j))^p 2^{ip}\frac{|A^i|}{|B^i|}.
\end{split}
\end{equation}
\textbf{Estimate of $J$:} By Lemma \ref{energyest}, we obtain a constant $C=C(p,\Lambda)$ such that
\begin{equation}\label{tite}
\begin{split}
\int_{B^i}|\nabla w_i|^{p}\psi_{i}^p\,dx&\leq C(J_1+J_2+J_3),
\end{split}
\end{equation}
where
$$
J_1=\int_{B^i}w_{i}^p |\nabla\psi_i|^p\,dx,\quad J_2=\int_{B^i}\int_{B^i}{\max\{w_i(x),w_i(y)\}^p|\psi_i(x)-\psi_i(y)|^p}\,d\mu
$$
and
$$
J_3=\esssup_{x\in\hat{B}^i}\int_{{\mathbb{R}^n\setminus B^i}}{\frac{w_i(y)^{p-1}}{|x-y|^{n+ps}}}\,dy
\cdot\int_{B^i}w_i \psi_{i}^p\,dx.
$$
From \eqref{Jite} we have 
\begin{equation}\label{j1}
J_1\leq cr_{j+1}^{-p}(\varepsilon\omega(r_j))^p 2^{ip}|A^i|,
\end{equation}
with $c=c(n,p)>0$. For $x\in\hat{B^i}$ and $y\in\R^n\setminus B^{i}$, we have
\begin{equation}\label{nee1}
\begin{split}
\frac{1}{|y-x|}
&=\frac{1}{|y-x_0|}\frac{|y-x_0|}{|y-x|}
\leq\frac{1}{|y-x_0|}\Big(1+\frac{|x-x_0|}{|y-x|}\Big)\\
&\leq\frac{1}{|y-x_0|}\Big(1+\frac{\hat{\rho}_i}{\rho_i -\hat{\rho}_i}\Big)
\leq\frac{2^{i+4}}{|y-x_0|}. 
\end{split}
\end{equation}
By applying \eqref{nee1}, \eqref{tailest1}, the properties of $\psi_i$, $r\in(0,1]$ and proceeding along the lines of the proof of the estimates $(5.12)$ and $(5.15)$ in \cite[Page 1297]{Kuusilocal}, we obtain
\begin{equation}\label{jmm}
J_{m}\leq Cr_{j+1}^{-p}(\varepsilon\omega(r_j))^p 2^{i(n+p)}|A^i|,
\quad m=2,3,
\end{equation}
for some positive constant $C$ depending on $n,p,s,\Lambda$ and the difference between $\frac{p}{p-1}$ and $\alpha$. Using \eqref{j1} and \eqref{jmm} in \eqref{tite}, we obtain
\begin{equation}\label{Iite}
\begin{split}
J&=\fint_{B^i}|\nabla w_i|^{p}\psi_{i}^p\,dx\leq Cr_{j+1}^{-p}(\varepsilon\omega(r_j))^p 2^{i(n+p)}\frac{|A^i|}{|B^i|},
\end{split}
\end{equation}
for some positive constant $C$ depending only on $n,p,s,\Lambda$ and the difference between $\frac{p}{p-1}$ and $\alpha$. 
Let 
\[
Y_i=\frac{|A^i|}{|B^i|}, 
\quad i=0,1,2,\dots.
\] 
Noting that $k_i- k_{i+1}=2^{-i-1}\varepsilon \omega(r_j)$ and applying \eqref{Jite} and \eqref{Iite} in \eqref{Poincareless1}, we get
\[
Y_{i+1}\leq C 2^{i(2p+n)\kappa}Y_{i}^{\kappa},
\]
for some constant $C$ depending only on $n,p,s,\Lambda$ and the difference between $\frac{p}{p-1}$ and $\alpha$. From Step $1$, by \eqref{s1}, we have 
$$
Y_0\leq\frac{\hat{C}}{\log(\frac{1}{\eta})},
$$
for some positive constant $\hat{C}$ depending only on $n,p,s,\Lambda$ and the difference between $\frac{p}{p-1}$ and $\alpha$. Let 
$$
c_0=C,\quad b=2^{(2p+n)\kappa},\quad \beta=\kappa-1
\quad\text{and}\quad \eta_1=c_{0}^{-\frac{1}{\beta}}b^{-\frac{1}{\beta^2}}.
$$
By choosing $\eta=\frac{1}{2} \min\big\{\frac{1}{4},e^{-\frac{\hat{C}}{\eta_1}}\big\}$ we have $Y_0\leq\eta_1$. 
Thus by Lemma \ref{iteration} we deduce that $\lim_{i\to\infty}Y_i=0$ and therefore, 
$u_j\geq\varepsilon\omega(r_j)\text{ in }B_{j+1}$.
Using the definition of $u_j$ from \eqref{uj}, we obtain
\begin{equation}\label{oscj1}
\text{osc}_{B_{j+1}}\,u\leq (1-\varepsilon)\omega(r_j)=(1-\varepsilon)\eta^{-\alpha}\omega(r_{j+1})\leq\omega(r_{j+1}),
\end{equation}
where we have chosen $\alpha\in(0,\frac{p}{p-1})$ (depending on $n,p,s,\Lambda$) small enough such that
\[
\eta^\alpha\geq 1-\varepsilon=1-\eta^{\frac{p}{p-1}-\alpha}.
\]
Thus \eqref{oscj1} proves the induction estimate \eqref{oscest} for $i=j+1$. Hence the result follows.
\end{proof}

\section{Tail estimate}

The following tail estimate will be useful for us.

\begin{Lemma}\label{Tail}
Let $u$ be a weak solution of \eqref{maineqn} such that $u\geq 0$ in $B_R(x_0)\subset\Omega$.
There exists a positive constant $c=c(n,p,s,\Lambda)$ such that
\begin{equation}\label{tailest}
\Tail(u_{+};x_0,r)\leq c\esssup_{B_r(x_0)}\,u+c\Big(\frac{r}{R}\Big)^\frac{p}{p-1}\Tail(u_{-};x_0,R),
\end{equation}
whenever $0<r<R$ with $r\in(0,1]$.
Here $\Tail(\cdot)$ is given by \eqref{loctail}.
\end{Lemma}

\begin{proof}
Let $M=\esssup_{B_r(x_0)}\,u$ and $\psi\in C_c^{\infty}(B_r(x_0))$ be a cutoff function such that
$0\leq\psi\leq 1$ in $B_r(x_0)$, $\psi=1$ in $B_{\frac{r}{2}}(x_0)$ and $|\nabla\psi|\leq\frac{8}{r}$ in $B_r(x_0)$.
By letting $w=u-2M$ and choosing $\phi=w\psi^p$ as a test function in \eqref{weaksubsupsoln} we obtain
\begin{equation}\label{tailtest}
\begin{split}
0&=\int_{B_r(x_0)}|\nabla u|^{p-2}\nabla u\cdot\nabla(w\psi^p)\,dx\\
&\qquad+\int_{B_r(x_0)}\int_{B_r(x_0)}\mathcal{A}(u(x,y))(w(x)\psi(x)^p-w(y)\psi(y)^p)\,d\mu\\
&\qquad+2\int_{B_r(x_0)}\int_{\mathbb{R}^n\setminus B_r(x_0)}\mathcal{A}(u(x,y))w(x)\psi(x)^p\,d\mu\\
&=I_1+I_2+I_3.
\end{split}
\end{equation}
\textbf{Estimate of $I_1$:} By Young's inequality, the estimate
\begin{equation*}
\begin{split}
|\nabla w|^{p-2}\nabla w\cdot\nabla(w\psi^p)
&=|\nabla w|^p \psi^p+p\psi^{p-1}w|\nabla w|^{p-2}\nabla w\cdot\nabla\psi\\
&\geq\frac{1}{2}|\nabla w|^p\psi^p-c(p)|w|^p|\nabla\psi|^p
-c(p)M^p|\nabla\psi|^p,
\end{split}
\end{equation*}
holds in $B_r(x_0)$.
By the properties of $\psi$, we have
\begin{equation}\label{tailestI1}
I_1=\int_{B_r(x_0)}|\nabla u|^{p-2}\nabla u\cdot\nabla(w\psi^p)\,dx\geq -c(p)M^p r^{-p}|B_r(x_0)|.
\end{equation}
\textbf{Estimate of $I_2$ and $I_3$:}
Proceeding along the lines of the proof of the estimates $(4.11)$ and $(4.9)$ in \cite[Pages 1827--1828]{KuusiHarnack} and using the fact that $r\in(0,1]$, 
we obtain a constant $c=c(n,p,s,\Lambda)>0$ such that
\begin{equation}\label{tailestI2}
\begin{split}
I_2&=\int_{B_r(x_0)}\int_{B_r(x_0)}\mathcal{A}(u(x,y))(w(x)\psi(x)^p-w(y)\psi(y)^p)\,d\mu\geq -cM^p r^{-p}|B_r(x_0)|,
\end{split}
\end{equation}
and
\begin{equation}\label{tailestI3}
\begin{split}
I_3&=2\int_{B_r(x_0)}\int_{\mathbb{R}^n\setminus B_r(x_0)}\mathcal{A}(u(x,y))w(x)\psi(x)^p\,d\mu
\geq cM r^{-p}\Tail(u_{+};x_0,r)^{p-1}|B_r(x_0)|\\
&\qquad-cMR^{-p}\Tail(u_{-};x_0,R)^{p-1}|B_r(x_0)|-cM^p r^{-p}|B_r(x_0)|. 
\end{split}
\end{equation}
The estimate in \eqref{tailest} follows by applying \eqref{tailestI1}, \eqref{tailestI2} and \eqref{tailestI3} in \eqref{tailtest}.
\end{proof}

\section{Expansion of positivity}

The following lemma shows that the expansion of positivity technique applies to mixed problems.

\begin{Lemma}\label{DGLemma}
Let $u$ be a weak supersolution of  \eqref{maineqn} such that $u\geq 0$ in $B_R(x_0)\subset\Omega$. 
Assume $k\geq 0$ and there exists $\tau\in(0,1]$ such that
\begin{equation}\label{expangiven}
\big|B_r(x_0)\cap\{u\geq k\}\big|\geq \tau|B_r(x_0)|,
\end{equation}
for some $r\in(0,1]$ with $0<r<\frac{R}{16}$. There exists a constant $\delta=\delta(n,p,s,\Lambda,\tau)\in(0,\frac{1}{4})$ such that
\begin{equation}\label{expan}
\essinf_{B_{4r}(x_0)}\,u\geq\delta k-\Big(\frac{r}{R}\Big)^\frac{p}{p-1}\Tail(u_{-};x_0,R),
\end{equation} 
where $\Tail(\cdot)$ is given by \eqref{loctail}.
\end{Lemma}
\begin{proof}
We prove the lemma in two steps.\\
\textbf{Step 1.} Under the assumption in  \eqref{expangiven}, we claim that there exists a positive constant $c_1=c(n,p,s,\Lambda)$ such that 
\begin{equation}\label{expanstep1}
\Big|B_{6r}(x_0)\cap \Big\{u\leq 2\delta k-\frac{1}{2}\Big(\frac{r}{R}\Big)^\frac{p}{p-1}\Tail(u_{-};x_0,R)-\epsilon\Big\}\Big|
\leq \frac{c_1}{\tau\log\frac{1}{2\delta}}|B_{6r}(x_0)|
\end{equation}
for every $\delta\in(0,\frac{1}{4})$ and for every $\epsilon>0$.

Let $\epsilon>0$ and $\psi\in C_c^{\infty}(B_{7r}(x_0))$ be a cutoff function such that 
$0\leq\psi\leq 1$ in $B_{7r}(x_0)$, $\psi=1$ in $B_{6r}(x_0)$ and $|\nabla\psi|\leq\frac{8}{r}$ in $B_{7r}(x_0)$.
We denote $w=u+t_{\epsilon}$, where
$$
t_{\epsilon}=\frac{1}{2}\Big(\frac{r}{R}\Big)^\frac{p}{p-1}\Tail(u_{-};x_0,R)+\epsilon.
$$
Since $w$ is a weak supersolution of \eqref{maineqn}, we can choose $\phi=w^{1-p}\psi^p$ as a test function in \eqref{weaksubsupsoln} to obtain
\begin{equation}\label{DGLtsteqn1}
\begin{split}
0&\leq\int_{B_{8r}(x_0)}|\nabla w|^{p-2}\nabla w\cdot\nabla( w^{1-p}\psi^p)\,dx\\
&\qquad+\int_{B_{8r}(x_0)}\int_{B_{8r}(x_0)}\mathcal{A}( w(x,y))( w(x)^{1-p}\psi(x)^p- w(y)^{1-p}\psi(y)^p)\,d\mu\\
&\qquad+2\int_{\mathbb{R}^n\setminus B_{8r}(x_0)}\int_{B_{8r}(x_0)}\mathcal{A}( w(x,y)) w(x)^{1-p}\psi(x)^p\,d\mu\\
&=I_1+I_2+I_3.
\end{split}
\end{equation}
\textbf{Estimate of $I_1$:} Proceeding similarly as in the proof of \cite[Pages 717--718, Lemma 3.4]{Kin-Kuusi} and using the properties of $\psi$, we obtain a constant $c=c(p)>0$ such that
\begin{equation}\label{EstI1}
\begin{split}
I_1&=\int_{B_{8r}(x_0)}|\nabla w|^{p-2}\nabla w\cdot\nabla( w^{1-p}\psi^p)\,dx\leq -c\int_{B_{6r}(x_0)}|\nabla\log w|^p\,dx+cr^{n-p}.
\end{split}
\end{equation}
\textbf{Estimate of $I_2$:} Arguing as in the proof of the estimate of $I_1$ in \cite[page 1817]{KuusiHarnack} and using the fact that $r\in(0,1]$, we
obtain a constant $c=c(n,p,s,\Lambda)>0$ such that
\begin{equation}\label{EstI2}
\begin{split}
I_2&=\int_{B_{8r}(x_0)}\int_{B_{8r}(x_0)}\mathcal{A}( w(x,y))( w(x)^{1-p}\psi(x)^p- w(y)^{1-p}\psi(y)^p)\,d\mu\\
&\leq -\frac{1}{c}\int_{B_{6r}(x_0)}\int_{B_{6r}(x_0)}\bigg|\log\Big(\frac{ w(x)}{ w(y)}\Big)\bigg|^p\,d\mu+c r^{n-p}.
\end{split}
\end{equation}
\textbf{Estimate of $I_3$:} Here we follow the proof of the estimate of $I_2$ in \cite[Pages 1817--1818]{KuusiHarnack}. To this end, we write
\begin{equation}\label{estI3l}
\begin{split}
I_3&=2\int_{\mathbb{R}^n\setminus B_{8r}(x_0)}\int_{B_{8r}(x_0)}\mathcal{A}( w(x,y)) w(x)^{1-p}\psi(x)^p\,d\mu=2(I_3^{1}+I_3^{2}),
\end{split}
\end{equation}
where
$$
I_3^{1}=\int_{\mathbb{R}^n\setminus B_{8r}(x_0)\cap \{w(y)<0\}}\int_{B_{8r}(x_0)}\mathcal{A}( w(x,y)) w(x)^{1-p}\psi(x)^p\,d\mu
$$
and
$$
I_3^{2}=\int_{\mathbb{R}^n\setminus B_{8r}(x_0)\cap \{w(y)\geq 0\}}\int_{B_{8r}(x_0)}\mathcal{A}( w(x,y)) w(x)^{1-p}\psi(x)^p\,d\mu.
$$
\textbf{Estimate of $I_3^{1}$:} Using the definitions of $w$ and $t_{\epsilon}$, the assumption on the kernel $K$ and the fact that support of $\psi$ is inside $B_{7r}(x_0)$, we get
\begin{equation}\label{l11}
\begin{split}
I_3^{1}&\leq c r^n\int_{\mathbb{R}^n\setminus B_{8r}(x_0)}\Big(1+\frac{(w(y))_-}{t_{\epsilon}}\Big)^{p-1}|y-x_0|^{-n-ps}\,dy\\
&\leq cr^{n-ps}+cr^{n}t_{\epsilon}^{1-p}R^{-p}\mathrm{Tail}(u_-;x_0,R)^{p-1}\leq cr^{n-p},
\end{split}
\end{equation}
with $c=c(n,p,s,\Lambda)$. Here we also used the hypothesis that $u\geq 0$ in $B_R(x_0)$ and $r\in(0,1]$.\\
\textbf{Estimate of $I_3^{2}$:} Let $x\in B_{8r}(x_0)$. Suppose $y\in \mathbb{R}^n\setminus B_{8r}(x_0)$ such that $w(y)\geq 0$. If $w(x)-w(y)\leq 0$, then $\mathcal{A}(w(x,y))\leq 0$. If $w(x)-w(y)\geq 0$, then $\mathcal{A}(w(x,y))\leq w(x)^{p-1}$. Therefore, again using the assumption on the kernel $K$ and the fact that $\psi$ is supported in $B_{7r}(x_0)$ along with $r\in(0,1]$, we get
\begin{equation}\label{l12}
I_3^{2}\leq c\int_{\mathbb{R}^n\setminus B_{8r}(x_0)}\int_{B_{7r}(x_0)}|y-x_0|^{-n-ps}\,dx\,dy\leq cr^{n-p},
\end{equation}
with $c=c(n,p,s,\Lambda)$. By applying \eqref{l11} and \eqref{l12} in \eqref{estI3l}, we have 
\begin{equation}\label{estI3}
I_3\leq cr^{n-p},
\end{equation}
with $c=c(n,p,s,\Lambda)$.\\
By using \eqref{EstI1}, \eqref{EstI2} and \eqref{estI3} in \eqref{DGLtsteqn1}, we obtain
\begin{equation}\label{estI123}
\begin{split}
\int_{B_{6r}(x_0)}|\nabla\log w|^p\,dx
+\int_{B_{6r}(x_0)}\int_{B_{6r}(x_0)}\bigg|\log\Big(\frac{ w(x)}{ w(y)}\Big)\bigg|^p\,d\mu\leq cr^{n-p},
\end{split}
\end{equation}
for some constant $c=c(n,p,s,\Lambda)$. For $\delta\in(0,\frac{1}{4})$, we denote
$$
v=\biggl(\min\biggl\{\log\frac{1}{2\delta},\log\frac{k+t_{\epsilon}}{ w}\biggr\}\biggr)_{+}.
$$
By \eqref{estI123}, we have
\begin{equation}\label{pcase}
\int_{B_{6r}(x_0)}|\nabla v|^p\,dx
\leq\int_{B_{6r}(x_0)}|\nabla\log  w|^p\,dx\leq c r^{n-p}.
\end{equation}
From \eqref{pcase}, by H\"older's inequality and Poincar\'e inequality (see \cite[Theorem 2]{Evans}), we obtain
\begin{equation}\label{pcase1}
\int_{B_{6r}(x_0)}|v-(v)_{B_{6r}(x_0)}|\,dx
\leq c r^{1+\frac{n}{p'}}\bigg(\int_{B_{6r}(x_0)}|\nabla v|^p\,dx\bigg)^\frac{1}{p}\leq c|B_{6r}(x_0)|,
\end{equation}
where $p'=\frac{p}{p-1}$ and $(v)_{B_{6r}(x_0)}=\fint_{B_{6r}(x_0)}v\,dx$. We observe that
$\{v=0\}=\{ w\geq k+t_{\epsilon}\}=\{u\geq k\}$.
By the assumption \eqref{expangiven}, it follows that
\begin{equation}\label{useofgiven}
|B_{6r}(x_0)\cap\{v=0\}|\geq\frac{\tau}{6^n}|B_{6r}(x_0)|.
\end{equation}
Following the proof of \cite[Page 1819, Lemma 3.1]{KuusiHarnack} and using \eqref{useofgiven}, we obtain
\begin{equation}\label{pcase2}
\begin{split}
\log\,\frac{1}{2\delta}
&=\frac{1}{|B_{6r}(x_0)\cap\{v=0\}|}\int_{B_{6r}(x_0)\cap\{v=0\}}\Big(\log\frac{1}{2\delta}-v(x)\Big)\,dx\\
&\leq\frac{6^n}{\tau}\Big(\log\frac{1}{2\delta}-(v)_{B_{6r}}\Big).
\end{split}
\end{equation}
Now integrating \eqref{pcase2} over the set $B_{6r}(x_0)\cap\{v=\log\frac{1}{2\delta}\}$ and using \eqref{pcase1}, we obtain a constant $c_1=c_1(n,p,s,\Lambda)$ such that 
$$
\Big|\Big\{v=\log\frac{1}{2\delta}\Big\}\cap B_{6r}(x_0)\Big|\log\frac{1}{2\delta}\leq\frac{6^n}{\tau}\int_{B_{6r}(x_0)}|v-(v)_{B_{6r}(x_0)}|\,dx
\leq\frac{c_1}{\tau}|B_{6r}(x_0)|.
$$
Hence, for any $\delta\in(0,\frac{1}{4})$, we have
$$
\big|B_{6r}(x_0)\cap\{ w\leq 2\delta(k+t_{\epsilon})\}\big|\leq\frac{c_1}{\tau}\frac{1}{\log\frac{1}{2\delta}}|B_{6r}(x_0)|.
$$
This implies \eqref{expanstep1}.\\
\textbf{Step 2.} We claim that, for every $\epsilon>0$, there exists a constant $\delta=\delta(n,p,s,\Lambda,\tau)\in(0,\frac{1}{4})$ such that
\begin{equation}\label{expanaux}
\essinf_{B_{4r}(x_0)}\,u\geq\delta k-\Big(\frac{r}{R}\Big)^\frac{p}{p-1}\Tail(u_{-};x_0,R)-2\epsilon.
\end{equation}
 As a consequence of \eqref{expanaux}, the property \eqref{expan} follows.

To prove \eqref{expanaux}, without loss of generality, we may assume that
\begin{equation}\label{assume}
\delta k\geq\Big(\frac{r}{R}\Big)^\frac{p}{p-1}\Tail(u_{-};x_0,R)+2\epsilon.
\end{equation}
Otherwise \eqref{expanaux} holds tue, since $u\geq 0$ in $B_R(x_0)$.

Let $\rho\in[r,6r]$ and $\psi\in C_c^{\infty}(B_{\rho}(x_0))$ be a cutoff function such that $0\leq\psi\leq 1$ in $B_{\rho}(x_0)$.
For any $l\in(\delta k,2\delta k)$, from Lemma \ref{energyest} and the proof of \cite[Pages 1820--1821, Lemma 3.2]{KuusiHarnack} for $w=(l-u)_{+}$, 
for some constant $c=c(n,p,s,\Lambda)$, we obtain
\begin{equation}\label{energyapp}
\begin{split}
&\int_{B_{\rho}(x_0)}\psi^p|\nabla\,w|^p\,dx
+\int_{B_{\rho}(x_0)}\int_{B_{\rho}(x_0)}|w(x)\psi(x)-w(y)\psi(y)|^p\,d\mu\\
&\leq c\int_{B_{\rho}(x_0)}w^p\,|\nabla\psi|^p\,dx+c\int_{B_{\rho}(x_0)}\int_{B_{\rho}(x_0)}\max\{w(x),w(y)\}^p|\psi(x)-\psi(y)|^p\,d\mu\\
&\qquad+cl\esssup_{x\in\supp\psi}\int_{{\mathbb{R}^n}\setminus B_{\rho}(x_0)}\big(l+(u(y))_{-}\big)^{p-1}\,K(x,y)\,dy
\cdot|B_{\rho}(x_0)\cap\{u<l\}|\\
&=J_1+J_2+J_3.
\end{split}
\end{equation}
We apply Lemma \ref{iteration} to conclude the proof. For $j=0,1,2,\dots$, we denote
\begin{equation}\label{parameter}
l=k_j=\delta k+2^{-j-1}\delta k,
\quad
\rho=\rho_j=4r+2^{1-j}r,
\quad
\hat{\rho_j}=\frac{\rho_j+\rho_{j+1}}{2}.
\end{equation}
Then $l\in(\delta k,2\delta k)$, $\rho_j,\hat{\rho_j}\in(4r,6r)$ and 
$$
k_j-k_{j+1}=2^{-j-2}\delta k\geq 2^{-j-3}k_j
$$
for every $j=0,1,2,\dots$.
Set $B_j=B_{\rho_j}(x_0),\,\hat{B}_j=B_{\hat{\rho}_j}(x_0)$ and we observe that
$$
w_j=(k_j-u)_{+}\geq 2^{-j-3} k_j\chi_{\{u<k_{j+1}\}}.
$$
Let $(\psi_j)_{j=0}^{\infty}\subset C_c^{\infty}(\hat{B}_j)$ be a sequence of cutoff functions such that
$0\leq\psi_j\leq 1$ in $\hat{B}_j$, $\psi_j=1$ in $B_{j+1}$ and $|\nabla\psi_j|\leq\frac{2^{j+3}}{r}$.
We choose $\psi=\psi_j$, $w=w_j$ in \eqref{energyapp}. By the properties of $\psi_j$, we obtain
\begin{equation}\label{I1jest}
J_1=\int_{B_j}w_j^p|\nabla\psi_j|^p\,dx\leq{c(p)2^{jp}}k_j^{p}r^{-p}|B_j\cap\{u<k_j\}|.
\end{equation}
Now proceeding along the lines of the the proof of \cite[Page 1822, Lemma 3.2]{KuusiHarnack}, for any $r\in(0,1]$, we get 
\begin{equation}\label{I2jest}
\begin{split}
J_2&=\int_{B_j}\int_{B_j}\max\{w_j(x),w_j(y)\}^p|\psi_j(x)-\psi_j(y)|^p\,d\mu\\
&\leq c(n,p,s,\Lambda)2^{jp} k_j^{p}r^{-p}|B_j\cap\{u<k_j\}|.
\end{split}
\end{equation}
To estimate $J_3$, we follow the proof of \cite[Page 1823, Lemma 3.2]{KuusiHarnack}. To this end, we observe that, for any $x\in\mathrm{supp}\,\psi_j\subset\hat{B}_j$ and $y\in\mathbb{R}^n\setminus B_j$, we have
\begin{equation}\label{ne}
\frac{|y-x_0|}{|y-x|}=\frac{|y-x+x-x_0|}{|y-x|}\leq 1+\frac{|x-x_0|}{|y-x|}\leq 1+\frac{\hat{\rho}_j}{\rho_j -\hat{\rho}_j}=2^{j+4}.
\end{equation}
Using \eqref{ne} and the properties of the kernel $K$, we have
\begin{equation}\label{ne1}
\begin{split}
&\esssup_{x\in\supp\psi_j}\int_{\mathbb{R}^n\setminus B_j}\big(k_j+(u(y))_{-}\big)^{p-1}\,K(x,y)\,dy\\
&\qquad\leq c2^{j(n+ps)}\int_{\mathbb{R}^n\setminus B_j}\big(k_j+(u(y))_{-}\big)^{p-1}|y-x_0|^{-n-ps}\,dy\\
&\qquad\leq c2^{j(n+ps)}\bigg(k_j^{p-1}r^{-ps}+\int_{\mathbb{R}^n\setminus B_R(x_0)}(u(y))_{-}^{p-1}|y-x_0|^{-n-ps}\,dy\bigg)\\
&\qquad= c2^{j(n+ps)}\Big(k_j^{p-1}r^{-p}+r^{-p}\big(\frac{r}{R}\big)^p \mathrm{Tail}(u_-;x_0,R)^{p-1}\Big)\\
&\qquad\leq c2^{j(n+ps)}k_j^{p-1}r^{-p},
\end{split}
\end{equation}
with $c=c(n,p,s,\Lambda)$.
Here we also used the fact that $r\in(0,1]$ along with \eqref{assume}, $\delta k<k_j$ and the fact that $u\geq 0$ in $B_R(x_0)$.
Therefore, from \eqref{ne1}, we obtain
\begin{equation}\label{I3jest}
\begin{split}
J_3&=ck_j\esssup_{x\in\supp\psi_j}\int_{\mathbb{R}^n\setminus B_j}\big(k_j+(u(y))_{-}\big)^{p-1}\,K(x,y)\,dy
\cdot|B_j\cap\{u<k_j\}|\\
&\leq c2^{j(n+ps)}k_j^{p}r^{-p}|B_j\cap\{u<k_j\}|,
\end{split}
\end{equation}
with $c=c(n,p,s,\Lambda)$.
By using \eqref{I1jest}, \eqref{I2jest} and \eqref{I3jest} in \eqref{energyapp}, we obtain
\begin{equation}\label{J123est}
\int_{B_j}\psi_j^p|\nabla w_j|^p\,dx\leq c 2^{j(n+ps+p)}k_j^{p}r^{-p}|B_j\cap\{u<k_j\}|,
\end{equation}
with $c=c(n,p,s,\Lambda)$.
By applying the Sobolev inequality in \eqref{e.friedrich} along with \eqref{I1jest} and \eqref{J123est}, for $\kappa$ defined in \eqref{kappa}, we obtain a constant $c=c(n,p,s,\Lambda)$ such that
\begin{equation}\label{Poincareless}
\begin{split}
(k_j-k_{j+1})^p&\bigg(\frac{|B_{j+1}\cap\{u<k_{j+1}\}|}{|B_{j+1}|}\bigg)^\frac{1}{\kappa}
\leq\bigg(\fint_{B_{j+1}}w_j^{\kappa p}\psi_j^{\kappa p}\,dx\bigg)^\frac{1}{\kappa}\leq c \bigg(\fint_{B_{j}}w_j^{\kappa p}\psi_j^{\kappa p}\,dx\bigg)^\frac{1}{\kappa}\\
&\leq c r^{p}\fint_{B_j}|\nabla(w_j \psi_j)|^p\,dx\leq c 2^{j(n+ps+p)} k_j^{p}\frac{|B_j\cap\{u<k_j\}|}{|B_j|}.
\end{split}
\end{equation}
Let
$$
Y_j=\frac{|B_j\cap\{u<k_j\}|}{|B_j|},
\quad j=0,1,2,\dots.
$$
From \eqref{Poincareless} we have 
\begin{equation}\label{Iteapp}
Y_{j+1}\leq c_2\, 2^{j(n+2p+ps)\kappa} Y_j^{\kappa},
\quad j=0,1,2,\dots,
\end{equation}
for some constant $c_2=c_2(n,p,s,\Lambda)$.
We choose
$c_0=c_2$, $b=2^{(n+2p+ps)\kappa}>1$ and $\beta=\kappa-1>0$
in Lemma \ref{iteration}. By \eqref{assume}, we have
$$
k_0=\frac{3}{2}\delta k\leq 2\delta k-\frac{1}{2}\Big(\frac{r}{R}\Big)^\frac{p}{p-1}\Tail(u_{-};x_0,R)-\epsilon.
$$
By \eqref{expanstep1} we have
\begin{equation}\label{iteini}
Y_0\leq
\frac{\big|B_{6r}(x_0)\cap\big\{u\leq 2\delta k-\big(\frac{r}{R}\big)^\frac{p}{p-1}\Tail(u_{-};x_0,R)-\epsilon\big\}\big|}
{|B_{6r}(x_0)|}
\leq\frac{c_1}{\tau\log\frac{1}{2\delta}}
\end{equation}
for some constant $c_1=c_1(n,p,s,\Lambda)$ and for every $\delta\in(0,\frac{1}{4})$.
Using \eqref{iteini} we choose $\delta=\delta(n,p,s,\Lambda,\tau)\in(0,\frac{1}{4})$ such that
$$
0<\delta=\frac{1}{4}\exp\bigg(-\frac{c_1 c_0^\frac{1}{\beta} b^\frac{1}{\beta^2}}{\tau}\bigg)<\frac{1}{4},
$$
so that the estimate $Y_0\leq c_0^{-\frac{1}{\beta}}b^{-\frac{1}{\beta^2}}$ holds. 
By Lemma \ref{iteration} we conclude that $Y_j\to0$ as $j\to\infty$.
Therefore, we have 
\[
\essinf_{B_{4r}(x_0)}\,u\geq \delta k,
\] 
which gives \eqref{expanaux} and so \eqref{expan} holds.
\end{proof}

\section{Harnack inequalities}

Proceeding similarly as in the proof of \cite[Lemma 4.1]{KuusiHarnack}, along with an application of Lemma \ref{DGLemma}, we obtain the following preliminary version of 
the weak Harnack inequality, compared to Theorem \ref{thm3}.

\begin{Lemma}\label{WeakHarnacklemma}
Let $u$ be a weak supersolution of \eqref{maineqn} such that $u\geq 0$ in $B_R(x_0)\subset\Omega$. 
There exist constants $\eta=\eta(n,p,s,\Lambda)\in(0,1)$ and $c=c(n,p,s,\Lambda)\geq 1$ such that
\begin{equation}\label{wk}
\bigg(\fint_{B_r(x_0)}u^{\eta}\,dx\bigg)^\frac{1}{\eta}
\leq c\essinf_{B_r(x_0)}\,u+c\Big(\frac{r}{R}\Big)^\frac{p}{p-1}\Tail(u_{-};x_0,R),
\end{equation}
whenever $B_r(x_0)\subset B_R(x_0)$ with $r\in(0,1]$.
Here $\Tail(\cdot)$ is defined in \eqref{loctail}.
\end{Lemma}

The following Harnack inequality follows with a similar argument as in the proof of \cite[Theorem 1.1]{KuusiHarnack}. 
For convenience of the reader, we give a proof here in the mixed case. To this end, the following iteration lemma from \cite[Lemma 1.1]{GGacta} will be useful for us.
\begin{Lemma}\label{iteration1}
Let  $0\leq T_0\leq t\leq T_1$ and assume that $f:[T_1,T_2]\to[0,\infty)$ is a nonnegative bounded function. Suppose that for $T_0\leq t<s\leq T_1$, we have
\begin{equation}\label{itt}
f(t)\leq A(s-t)^{-\alpha} +B +\theta f(s),
\end{equation}
where $A,B,\alpha,\theta$ are nonegative constants and $\theta<1$. 
Then there exists a constant $c=c(\alpha,\theta)$ such that for every $\rho,R$ and $T_0\leq\rho<R\leq T_1$, we have
\begin{equation}\label{itt1}
f(\rho)\leq c(A(R-\rho)^{-\alpha}+B).
\end{equation}
\end{Lemma}

\begin{Theorem}\label{thm2}(\textbf{Harnack inequality}).
Let $u$ be a weak solution of \eqref{maineqn} such that $u\geq 0$ in $B_R(x_0)\subset\Omega$. 
There exists a positive constant $c=c(n,p,s,\Lambda)$ such that
\begin{equation}\label{Harnackest}
\esssup_{B_{\frac{r}{2}}(x_0)}\,u\leq c\essinf_{B_r(x_0)}\,u+c\Big(\frac{r}{R}\Big)^\frac{p}{p-1}\Tail(u_{-};x_0,R),
\end{equation}
whenever $B_r(x_0)\subset B_\frac{R}{2}(x_0)$ and $r\in(0,1]$.
Here $\Tail(\cdot)$ is given by \eqref{loctail}.
\end{Theorem}
\begin{proof}
Let $0<\rho<r$. Then by Lemma \ref{Solsubsup} and Theorem \ref{thm1}, for every $\delta\in(0,1]$, there exists a positive constant $c=c(n,p,s,\Lambda)$ such that
\begin{equation}\label{h1eqn}
\esssup_{B_{\frac{\rho}{2}}(x_0)}u\leq \delta\mathrm{Tail}(u_{+};x_0,\frac{\rho}{2})+c\delta^{-\gamma}\bigg(\fint_{B_{\rho}(x_0)}u_{+}^p\,dx\bigg)^\frac{1}{p},
\end{equation}
where $\gamma=\frac{(p-1)\kappa}{p(\kappa-1)}$ and $\kappa$ as in \eqref{kappa}. 
By Lemma \ref{Tail} and \eqref{h1eqn}, we obtain
\begin{equation}\label{h2eqn}
\esssup_{B_{\frac{\rho}{2}}(x_0)}u
\leq c\delta\Big(\esssup_{B_{\rho}(x_0)}u+\Big(\frac{\rho}{R}\Big)^\frac{p}{p-1}\mathrm{Tail}(u_{-};x_0,R)\Big)+c\delta^{-\gamma}\bigg(\fint_{B_{\rho}(x_0)}u_{+}^p\,dx\bigg)^\frac{1}{p},
\end{equation}
for some constant $c=c(n,p,s,\Lambda)$. Let $\frac{1}{2}\leq\sigma^{'}<\sigma\leq 1$ and $\rho=(\sigma-\sigma^{'})r$. 
Using a covering argument, it follows that
\begin{equation}\label{h2eqn1}
\begin{split}
\esssup_{B_{\sigma^{'}r}(x_0)}u
&\leq c\frac{\delta^{-\gamma}}{(\sigma-\sigma^{'})^{\frac{n}{p}}}\bigg(\fint_{B_{\sigma r}}u^p\,dx\bigg)^\frac{1}{p} 
+c\delta\esssup_{B_{\sigma r}(x_0)}\,u+c\delta\Big(\frac{r}{R}\Big)^\frac{p}{p-1}\mathrm{Tail}(u_{-};x_0,R)\\
&\leq c\frac{\delta^{-\gamma}}{(\sigma-\sigma^{'})^{\frac{n}{p}}}(\esssup_{B_{\sigma r}(x_0)}\,u)^\frac{p-t}{p}
\bigg(\fint_{B_{\sigma r}(x_0)}u^t\,dx\bigg)^\frac{1}{p}+c\delta\esssup_{B_{\sigma r}(x_0)}\,u\\
&\qquad+c\delta\Big(\frac{r}{R}\Big)^\frac{p}{p-1}\mathrm{Tail}(u_{-};x_0,R)
\end{split}
\end{equation}
for every $t\in(0,p)$ with a constant $c=c(n,p,s\Lambda)$. 
Young's inequality with exponents $\frac{p}{t}$ and $\frac{p}{p-t}$ and choosing $\delta=\frac{1}{4c}$ in \eqref{h2eqn1} implies that
\begin{equation}\label{h2eqn2}
\esssup_{B_{\sigma^{'}r}(x_0)}u\leq \frac{1}{2}\esssup_{B_{\sigma r}(x_0)}u
+\frac{c}{(\sigma-\sigma^{'})^\frac{n}{t}}\bigg(\fint_{B_{r}(x_0)}u^t\,dx\bigg)^\frac{1}{t}+c\Big(\frac{r}{R}\Big)^\frac{p}{p-1}\mathrm{Tail}(u_{-};x_0,R)
\end{equation}
for every $t\in(0,p)$ with a constant $c=c(n,p,s,t,\Lambda)$. 
Using Lemma \ref{iteration1} in \eqref{h2eqn2}, we have
\begin{equation}\label{h3eqn}
\begin{split}
\esssup_{B_{\frac{r}{2}}(x_0)}\,u
&\leq c\bigg(\fint_{B_{r}(x_0)}u^t\,dx\bigg)^\frac{1}{t}+c\Big(\frac{r}{R}\Big)^\frac{p}{p-1}\mathrm{Tail}(u_{-};x_0,R),
\end{split}
\end{equation}
for every $t\in(0,p)$ with a positive constant $c=c(n,p,s,t,\Lambda)$. Combining the above estimate \eqref{h3eqn} with Lemma \ref{WeakHarnacklemma} and choosing $t=\eta\in(0,1)$, the result follows.
\end{proof}

We have the following weak Harnack inequality for supersolutions of \eqref{maineqn}.

\begin{Theorem}\label{thm3}(\textbf{Weak Harnack inequality}).
Let $u$ be a weak supersolution of \eqref{maineqn} such that $u\geq 0$ in $B_R(x_0)\subset\Omega$. 
There exists a positive constant $c=c(n,p,s,\Lambda)$ such that
\begin{equation}\label{wkest}
\bigg(\fint_{B_{\frac{r}{2}}(x_0)}u^{l}\,dx\bigg)^\frac{1}{l}
\leq c\essinf_{B_{r}(x_0)}\,u+c\Big(\frac{r}{R}\Big)^\frac{p}{p-1}\Tail(u_{-};x_0,R),
\end{equation}
whenever $B_r(x_0)\subset B_{\frac{R}{2}}(x_0)$, $r\in(0,1]$ and  $0<l<\kappa (p-1)$.
Here $\kappa$ and $\Tail(\cdot)$ are given by \eqref{kappa} and \eqref{loctail}, respectively.
\end{Theorem}

\begin{proof}
We prove the result for $1<p<n$. For $p\geq n$, the result follows in a similar way. Let $r\in(0,1]$, $\frac{1}{2}<\tau'<\tau\leq\frac{3}{4}$ and we choose $\psi\in C_c^{\infty}(B_{\tau r}(x_0))$ such that $0\leq\psi\leq 1\text{ in }B_{\tau r}(x_0)$, $\psi=1$ in $B_{\tau' r}(x_0)$ and $|\nabla\psi|\leq\frac{4}{(\tau-\tau')r}$.
For $d>0$ and $q\in(1,p)$, we set 
\[
v=u+d
\quad\text{and}\quad 
w=(u+d)^\frac{p-q}{p}.
\]
Noting $r\in(0,1]$, the property of $\psi$ above and the proof of the estimate $(5.11)$ in \cite[Page 1834]{KuusiHarnack}, there exists a constant $c=c(n,p,s,\Lambda)$ such that
\begin{equation}\label{revestlocal}
I_1=\int_{B_r(x_0)}w^p|\nabla\psi|^p\,dx\leq\frac{c r^{-p}}{(\tau-\tau')^p}\int_{B_{\tau r}(x_0)}w^p\,dx,
\end{equation}
\begin{equation}\label{revestmax}
I_2=\int_{B_r(x_0)}\int_{B_r(x_0)}\max\{w(x),w(y)\}^p|\psi(x)-\psi(y)|^p\,d\mu
\leq\frac{c r^{-p}}{(\tau-\tau')^p}\int_{B_{\tau r}(x_0)}w^p\,dx.
\end{equation}

Assume that $\Tail(u_{-};x_0,R)$ is positive. Then for any $\epsilon>0$ and $r\in(0,1]$ choosing 
$$
d=\frac{1}{2}\Big(\frac{r}{R}\Big)^\frac{p}{p-1}\,\Tail(u_{-};x_0,R)+\epsilon>0,
$$
and noting that
\begin{equation}\label{sup}
\esssup_{z\in\supp\psi}\int_{\mathbb{R}^n\setminus B_r(x_0)}K(z,y)\,dy\leq c(n,p,s,\Lambda)r^{-p},
\end{equation}
we obtain
\begin{equation}\label{revestnonlocal}
\begin{split}
I_3&=\bigg(\esssup_{z\in\supp\psi}\int_{\mathbb{R}^n\setminus B_r(x_0)}K(z,y)\,dy+d^{1-p}R^{-p}\Tail(u_{-};x_0,R)^{p-1}\bigg)
\int_{B_r(x_0)}w^p\psi^p\,dx\\
&\leq\frac{c(n,p,s,\Lambda)r^{-p}}{(\tau-\tau')^p}\int_{B_{\tau r}(x_0)}w^p\,dx.
\end{split}
\end{equation}
If $\Tail(u_{-};x_0,R)=0$, we can choose an arbitrary $d=\epsilon>0$ and again using \eqref{sup} the estimate in \eqref{revestnonlocal} follows.
Now using Sobolev inequality in \eqref{e.friedrich} and the fact that $\psi\equiv 1$ in $B_{\tau' r}$, $r\in(0,1]$, along with Lemma \ref{energyforrev} and the estimates \eqref{revestlocal}, \eqref{revestmax} and \eqref{revestnonlocal}, we have for $p^*=\frac{np}{n-p}$ with $1<p<n$,
\begin{equation}\label{revSobloc}
\begin{split}
\bigg(\fint_{B_{\tau' r}(x_0)} v^\frac{n(p-q)}{n-p}\,dx\bigg)^\frac{p}{p^*}
&=\bigg(\fint_{B_{\tau' r}(x_0)}w^{p^*}\,dx\bigg)^\frac{p}{p^*}
\leq\bigg(\fint_{B_{\tau r}(x_0)}|w\psi|^{p^*}\,dx\bigg)^\frac{p}{p^*}\\
&\leq(\tau r)^{p-n}\int_{B_{\tau r}(x_0)}|\nabla\,(w\psi)|^p\,dx
\leq\frac{c}{(\tau-\tau')^p}\fint_{B_{\tau r}(x_0)}w^p\,dx,
\end{split}
\end{equation}
with $c=c(n,p,s,q,\Lambda)$. Using $q\in(1,p)$ and the Moser iteration technique as in \cite[Theorem 8.18]{GT2001} and \cite[Theorem 1.2]{Tru2}, we get
\begin{equation}\label{rev1}
\bigg(\fint_{B_{\frac{r}{2}}(x_0)} v^l\,dx\bigg)^\frac{1}{l}
\leq c\bigg(\fint_{B_\frac{3r}{4}(x_0)} v^{l'}\,dx\bigg)^\frac{1}{l'},
\quad 0<l'<l<\frac{n(p-1)}{n-p}.
\end{equation}
Let $\eta\in(0,1)$ be given by Lemma \ref{WeakHarnacklemma} and then choosing $l'=\eta\in(0,1)$ and observing that
$$
\bigg(\fint_{B_{\frac{r}{2}}(x_0)}{u}^l\,dx\bigg)^\frac{1}{l}
\leq\bigg(\fint_{B_{\frac{r}{2}}(x_0)} v^l\,dx\bigg)^\frac{1}{l},
$$
we obtain from \eqref{rev1}
\begin{equation}\label{weakHarnackfinalest}
\bigg(\fint_{B_{\frac{r}{2}}(x_0)}u^l\,dx\bigg)^\frac{1}{l}
\leq c\essinf_{B_r(x_0)}  v+c\Big(\frac{r}{R}\Big)^\frac{p}{p-1}\Tail(u_{-};x_0,R),
\end{equation}
for all $0<l<\frac{n(p-1)}{n-p}$.
For any $\epsilon>0$, choosing 
$$
d=\frac{1}{2}\Big(\frac{r}{R}\Big)^\frac{p}{p-1}\Tail(u_{-};x_0,R)+\epsilon,
$$
in \eqref{weakHarnackfinalest} and letting $\epsilon\to 0$, the result follows.
\end{proof}

\section{Semicontinuity}

Before stating our results on pointwise behavior, we discuss a result from Liao \cite{Liao}.
Let $u$ be a measurable function that is locally essentially bounded below in $\Omega$. 
Let $\rho\in(0,1]$ be such that $B_\rho(y)\subset\Omega$. Assume that $a,c\in(0,1)$, $M>0$ and
$\mu_-\leq\essinf_{B_\rho(y)}u$.
Following  \cite{Liao}, we say that $u$ satisfies the property $(\mathcal{D})$, if there exists a constant $\tau\in(0,1)$ depending on $a,M,\mu_-$ and other data 
(may depend on the partial differential equation and will be made precise in Lemma \ref{DGL}), but independent of $\rho$, such that
$$
|\{u\leq\mu_-+M\}\cap B_\rho(y)|\leq\tau|B_\rho(y)|,
$$
implies that $u\geq\mu_-+aM$ almost everywhere in $B_{c\rho}(y)$.

Moreover, for $u\in L^1_{\loc}(\Omega)$, we denote the set of Lebesgue points of $u$ by
$$
\mathcal{F}=\bigg\{x\in\Omega:|u(x)|<\infty,\,\lim_{r\to 0}\fint_{B_r(x)}|u(x)-u(y)|\,dy=0\bigg\}.
$$
Note that, by the Lebesgue differentiation theorem, $|\mathcal{F}|=|\Omega|$.

The following result follows from \cite[Theorem 2.1]{Liao}.
\begin{Lemma}\label{lscthm}
Let $u$ be a measurable function that is locally integrable and locally essentially bounded below in $\Omega$. Assume that $u$ satisfies the property $(\mathcal{D})$. 
Then $u(x)=u_*(x)$ for every $x\in\mathcal{F}$, where
\[
u_*(x)=\lim_{r\to 0}\essinf_{y\in B_r(x)}u(y).
\]
In particular, $u_*$ is a lower semicontinuous representative of $u$ in $\Omega$.
\end{Lemma}

Since $u$ is assumed to be locally essentially bounded below, the  lower semicontinuous regularization $u_*(x)$ is well defined at every point $x\in\Omega$.  
Our final regularity results stated are consequences of Lemma \ref{lscthm} and Lemma \ref{DGL} below.

\begin{Theorem}\label{lscthm1}(\textbf{Lower semicontinuity}).
Let $u$ be a weak supersolution of \eqref{maineqn}. 
Then 
\[
u(x)=u_*(x)=\lim_{r\to 0}\essinf_{y\in B_r(x)}u(y)
\]
for every $x\in\mathcal{F}$. 
In particular, $u_*$ is a lower semicontinuous representative of $u$ in $\Omega$. 
\end{Theorem}

As a Corollary Theorem \ref{lscthm1}, we have the following result.
\begin{Corollary}\label{uscthm}(\textbf{Upper semicontinuity}).
Let $u$ be a weak subsolution of  \eqref{maineqn}. 
Then 
\[
u(x)=u^*(x)=\lim_{r\to 0}\esssup_{y\in B_r(x)}u(y)
\]
for every $x\in\mathcal{F}$. 
In particular, $u^*$ is an upper semicontinuous representative of $u$ in $\Omega$.
\end{Corollary}

We prove a De Giorgi type lemma for weak supersolutions of \eqref{maineqn}.

\begin{Lemma}\label{DGL}
Let $u$ be a weak supersolution of \eqref{maineqn}. 
Let $M>0$, $a\in(0,1)$, $B_r(x_0)\subset\Omega$ with $r\in(0,1]$ and
$\mu_-\leq\essinf_{B_r(x_0)}\,u,\,\lambda_-\leq\essinf_{\R^n}\,u$.
There exists a constant $\tau=\tau(n,p,s,\Lambda,a,M,\mu_-,\lambda_-)\in(0,1)$ such that if
$$
|\{u\leq\mu_-+M\}\cap B_r(x_0)|\leq\tau|B_r(x_0)|,
$$
then $u\geq\mu_-+aM$ almost everywhere in $B_{\frac{3r}{4}(x_0)}$.
\end{Lemma}

\begin{proof}
Without loss of generality, we may assume that $x_0=0$. For $j=0,1,2,\dots$, we denote
\begin{equation}\label{parameter1}
\begin{split}
k_j&=\mu_{-}+aM+\frac{(1-a)M}{2^j},
\quad
\bar{k}_j=\frac{k_j+k_{j+1}}{2},\\
\quad
r_j&=\frac{3r}{4}+\frac{r}{2^{j+2}},
\quad\bar{r}_j=\frac{r_j+r_{j+1}}{2},
\end{split}
\end{equation}
$B_j=B_{r_j}(0)$, $\bar{B}_j=B_{\bar{r}_j}(0)$, $w_j=(k_j-u)_{+}$ and $\bar{w}_j=(\bar{k}_j-u)_{+}$.
We observe that $B_{j+1}\subset\bar{B}_j\subset B_j$, $\bar{k}_j<k_j$ and hence $\bar{w}_j\leq w_j$. 
Let $(\psi_{j})_{j=0}^{\infty}\subset C_{c}^{\infty}(\bar{B}_j)$ be a sequence of cutoff functions satisfying $0 \leq \psi_{j} \leq 1$ in $B_j$, $\psi_{j}= 1$ in $B_{j+1}$,
$|\nabla\psi_j|\leq\frac{2^{j+3}}{r}$.

By applying Lemma \ref{energyest} to $w_j$, we obtain
\begin{equation}\label{energyapp2}
\int_{B_j}|\nabla w_j|^p \psi_j^p\,dx\leq
C(n,p,s,\Lambda)(I_1+I_2+I_3),
\end{equation}
where
\[
I_1=\int_{B_{j}}\int_{B_{j}}{\max\{w_j(x),w_j(y)\}^p|\psi_{j}(x)-\psi_{j}(y)|^p}\,d\mu,
\quad I_2=\int_{B_j}w_j^p|\nabla\psi_j|^p\,dx
\]
and
\[
I_3=\esssup_{x\in\supp\psi_j}\int_{{\mathbb{R}^n\setminus B_{j}}}{\frac{w_j(y)^{p-1}}{|x-y|^{n+ps}}}\,dy
\cdot\int_{B_{j}}w_j\psi_j^p\,dx.
\]
Since $u\geq \lambda_-$ in $\R^n$, noting the definition of $k_j$ from above, we have
\begin{equation}\label{nelsc}
w_j=(k_j-u)_+\leq (M+\mu_--\lambda_-)_+=L\text{ in }\R^n. 
\end{equation}
Let $A_j=\{u<k_j\}\cap B_j$. We estimate the terms $I_j$, for $j=1,2,3$, separately.\\
\textbf{Estimate of $I_1$:} Using $w_j\leq L$ from \eqref{nelsc} along with $\frac{r}{2}<r_j<r$, $r\in(0,1]$ and the properties of $\psi_j$, we obtain
\begin{equation}\label{estI1}
I_1=\int_{B_{j}}\int_{B_{j}}{\max\{w_j(x),w_j(y)\}^p|\psi_{j}(x)-\psi_{j}(y)|^p}\,d\mu
\leq C\frac{2^{jp}}{r^{p}}L^p|A_j|,
\end{equation}
with $C=C(n,p,s,\Lambda)$.\\
\textbf{Estimate of $I_2$:} Using the properties of $\psi_j$ and the fact that $w_j\leq L$ from \eqref{nelsc}, we have
\begin{equation}\label{estI2}
I_2=\int_{B_{j}}w_j^p|\nabla \psi_j|^p\,dx\\
\leq C\frac{2^{jp}}{r^{p}}L^p|A_j|,
\end{equation}
with $C=C(n,p,s,\Lambda)$.\\
\textbf{Estimate of $I_3$:} For every $x\in\supp\psi_{j}$ and every $y\in\mathbb{R}^n\setminus B_j$, we observe that
\begin{equation}\label{I2}
\frac{1}{|x-y|}=\frac{1}{|y|}\frac{|x-(x-y)|}{|x-y|}\leq\frac{1}{|y|}(1+2^{j+5})\leq\frac{2^{j+6}}{|y|}.
\end{equation}
Then using $r_j>\frac{r}{2}$, $w_j\leq L$, $0\leq\psi_j \leq 1$, we obtain
\begin{equation}\label{estI3final}
I_3=\esssup_{x\in\supp\psi_j}\int_{{\mathbb{R}^n\setminus B_{j}}}{\frac{w_j(y)^{p-1}}{|x-y|^{n+ps}}}\,dy
\cdot\int_{B_{j}}w_j\psi_j^p\,dx
\leq C(n,p,s)\frac{2^{j(n+ps)}}{r^{p}}L^p|A_j|,
\end{equation}
for every $r\in(0,1]$.
By using \eqref{estI1}, \eqref{estI2} and \eqref{estI3final} in \eqref{energyapp2}, we have
\begin{equation}\label{energyapp4}
\begin{split}
\int_{B_j}|\nabla w_j|^p \psi_j^p\,dx\leq C\frac{2^{j(n+ps+p)}}{r^{p}}L^p|A_j|,
\end{split}
\end{equation}
with $C=C(n,p,s,\Lambda)$.
Noting that $B_{j+1}\subset\bar{B}_j\subset B_j$, $\bar{w}_j\leq w_j$ and using the Sobolev inequality in \eqref{e.friedrich}, we obtain  
\begin{equation}\label{dist}
\begin{split}
\frac{(1-a)M}{2^{j+2}}|A_{j+1}|&=\int_{A_{j+1}}(\bar{k}_j-k_{j+1})\,dx
\leq\int_{B_{j+1}}\bar{w}_j\,dx\\
&\leq\int_{B_{j+1}}w_j\,dx
\leq|A_j|^{1-\frac{1}{p\kappa}}\Big(\int_{B_j}w_j^{p\kappa}\psi_j^{p\kappa}\,dx\Big)^\frac{1}{p\kappa}\\
&\leq C r^{1+\frac{n}{p\kappa}-\frac{n}{p}}|A_j|^{1-\frac{1}{p\kappa}}\Big(\int_{B_j}|\nabla(w_j\psi_j)|^p\,dx\Big)^\frac{1}{p},
\end{split}
\end{equation}
with $C=C(n,p,s)$ and $\kappa$ as given in \eqref{kappa}.
By using  \eqref{energyapp4} together with the fact that $w_j\leq L$ and the properties of $\psi_j$ in \eqref{dist}, we get 
\begin{equation}\label{dist8}
\begin{split}
\frac{(1-a)M}{2^{j+2}}|A_{j+1}|\leq C(n,p,s,\Lambda) L r^{\frac{n}{p}(\frac{1}{\kappa}-1)}2^\frac{j(n+ps+p)}{p} |A_j|^{1-\frac{1}{p\kappa}+\frac{1}{p}}.
\end{split}
\end{equation}
Hence, we obtain from \eqref{dist8} that 
\begin{equation}\label{dist7}
\begin{split}
|A_{j+1}|&\leq \frac{C(n,p,s,\Lambda)L r^{\frac{n}{p}(\frac{1}{\kappa}-1)}}{(1-a)M}2^{j(2+\frac{n+ps}{p})}|A_j|^{1-\frac{1}{p\kappa}+\frac{1}{p}}.
\end{split}
\end{equation}
By dividing both sides of \eqref{dist7} with $|B_{j+1}|$ and noting that $|B_j|<2^{n}|B_{j+1}|$ together by $r_j<r$, we obtain
\begin{equation}\label{dist4}
\begin{split}
Y_{j+1}&\leq\frac{C(n,p,s,\Lambda)L}{(1-a)M}2^{j(2+\frac{n+ps}{p})}Y_j^{1+\frac{1}{p}(1-\frac{1}{\kappa})},
\end{split}
\end{equation}
where we denoted $Y_j=\frac{|A_j|}{|B_j|}$, $j=0,1,2,\dots$. By choosing
\begin{align*}
c_0&=\frac{C(n,p,s,\Lambda)L}{(1-a)M},
\quad
b=2^{(2+\frac{n+ps}{p})},
\quad
\beta=\frac{1}{p}\Big(1-\frac{1}{\kappa}\Big),
\quad
\tau=c_0^{-\frac{1}{\beta}}b^{-\frac{1}{\beta^2}}\in(0,1)
\end{align*}
in Lemma \ref{iteration} gives $Y_j\to0$ as $j\to\infty$, if $Y_0\leq\tau$.
This implies that  $u\geq\mu_{-}+aM$ almost everywhere in $B_{\frac{3r}{4}}(0)$.
\end{proof}

\medskip

\noindent
Prashanta Garain\\
Department of Mathematics\\
Ben-Gurian University of the Negev\\
P.O.B.-653\\
Be'er Sheva-8410501, Israel\\
Email: pgarain92@gmail.com\\

\noindent
Juha Kinnunen\\
Department of Mathematics\\
P.O. Box 11100\\
FI-00076 Aalto University, Finland\\
Email: juha.k.kinnunen@aalto.fi

\end{document}